\documentclass[12pt]{article}
\usepackage{amsmath,amsthm,amssymb,latexsym,color}
\usepackage{hyperref}

\theoremstyle{plain}
\newtheorem{theor}{Theorem}
\theoremstyle{remark}
\newtheorem{rem}{Remark}
\theoremstyle{plain}
\newtheorem{prop}[theor]{Proposition}
\newtheorem{cor}[theor]{Corollary}
\newtheorem{lemma}[theor]{Lemma}
\newtheorem*{problem}{Problem}

\def\R{{\mathbb R}}
\def\Prob{{\mathbb P}}
\def\N{{\mathbb N}}

\def\Net{{\mathcal N}}

\def\Exp{{\mathbb E}}
\def\M{{\mathcal M}}
\def\Med{{\rm Med}}
\def\Im{{\rm Im}}

\def\supp{{\rm supp}}

\def\Proj{{\rm P}}
\def\spn{{\rm span}}

\def\concf{{\mathcal Q}}

\def\rank{{\rm rank}}
\def\CovM{{\Sigma}}

\def\Sph{{\rm S}}
\def\Graph{{\mathcal G}}
\def\ColorC{{\mathcal C}}
\def\Class{{\mathcal F}}
\def\GL{{\rm GL}}
\def\Id{{\rm Id}}
\newcommand\kmax[2]{\,(#1)\textbf{-}\max\limits_{\ell\in #2}}

\begin{document}

\title{Sample covariance matrices of heavy-tailed distributions}
\author{Konstantin Tikhomirov\\
\small{Department of Math.\ and Stat.\ Sciences, University of Alberta, Canada}\\
\small{Email: ktikhomi@ualberta.ca}}

\maketitle

\begin{abstract}
Let $p>2$, $B\geq 1$, $N\geq n$ and let $X$ be a centered $n$-dimensional random vector
with the identity covariance matrix such that
$\sup\limits_{a\in\Sph^{n-1}}\Exp|\langle X,a\rangle|^p\leq B$.
Further, let $X_1,X_2,\dots,X_N$ be independent copies of $X$, and $\CovM_N:=\frac{1}{N}\sum_{i=1}^N X_i {X_i}^T$
be the sample covariance matrix. We prove that
$$K^{-1}\|\CovM_N-\Id_n\|_{2\to 2}\leq\frac{1}{N}\max\limits_{i\leq N}\|X_i\|^2
+\Bigl(\frac{n}{N}\Bigr)^{1-2/p}\log^4\frac{N}{n}+\Bigl(\frac{n}{N}\Bigr)^{1-2/\min(p,4)}$$
with probability at least $1-\frac{1}{n}$, where
$K>0$ depends only on $B$ and $p$. In particular, for all $p>4$ we obtain
a quantitative Bai--Yin type theorem.
\end{abstract}

\section{Introduction}

Estimation of the covariance matrix of a multidimensional distribution
is a standard problem in statistics.
Assume we have a centered $n$-dimensional random vector $X$
with an unknown covariance matrix $\CovM=\Exp XX^T$,
and $N$ independent copies of $X$ (a sample): $X_1,X_2,\dots,X_N$.
In general, the problem is to construct an estimator for $\CovM$ --- a function
of $X_1,X_2,\dots,X_N$ taking values in the set of $n\times n$ matrices, such that
for certain class of distributions the random matrix produced by the estimator is
close (in some sense) to the actual covariance matrix $\CovM$.
Various restrictions may be imposed on the distribution of $X$.
Recent developments in the subject showed that, under certain assumptions
on the moments of $1$-dimensional projections of $X$, together with some rather strong structural assumptions on $\CovM$,
it is possible to obtain a satisfactory estimator of $\CovM$ even when the size of the sample $N$
is much smaller than the dimension $n$.
There is a vast literature dealing with these questions, which, however, do not have direct connection
with our results. As an example of those developments, we refer to \cite{BL08a,BL08b}.

In this note, we consider the standard estimator ---
{\it the sample covariance matrix}, defined as $\CovM_N:=\frac{1}{N}\sum_{i=1}^N {X_i}{X_i}^T
=\frac{1}{N}{A_N}^T A_N$, where $A_N$ is the $N\times n$ random matrix with rows $X_1,X_2,\dots,X_N$.
The law of large numbers implies that for any distribution with a well-defined covariance matrix we have the convergence
$\CovM_N\stackrel{N\to\infty}{\longrightarrow}\CovM$ a.s.\ entry-wise, hence, in any operator norm.
The question is what size of the sample one should take to approximate the actual covariance matrix
by $\CovM_N$ with a given precision and probability.

Given a random vector $X$, by $F_X$ we shall denote the cdf of $X$.
Let $\Class$ be a class of $n$-dimensional centered distributions
which is closed under invertible linear transformations (i.e.\ $F_{T(X)}\in\Class$ whenever $F_X\in\Class$ and $T\in\GL_n(\R)$),
and let $\delta\in(0,1)$. We want to identify the number $N$ such that for any $F_X\in\Class$
with the covariance matrix $\CovM$,
and for corresponding sample $X_1,X_2,\dots,X_N$ we have
$\|\CovM_N-\CovM\|_{2\to 2}\leq\delta\|\CovM\|_{2\to 2}$ with probability close to one, where $\|\cdot\|_{2\to 2}$
denotes the spectral norm of a matrix, i.e.\ its largest singular value.
It can be easily shown that the question reduces to checking the relation
$\|\CovM_N-\Id_n\|_{2\to 2}\leq\delta$ for {\it isotropic} distributions from $\Class$ i.e.\
those having the identity covariance matrix.
Moreover, the last inequality is equivalent to
$$\sqrt{(1-\delta)N}\leq s_{\min}(A_N)\leq s_{\max}(A_N)\leq \sqrt{(1+\delta)N},$$
where $s_{\min}(A_N)$ and $s_{\max}(A_N)$ are the smallest and the largest singular values of $A_N$
given by
$s_{\min}(A_N):=\inf\limits_{a\in \Sph^{n-1}}\|A_N(a)\|$,
$s_{\max}(A_N):=\sup\limits_{a\in\Sph^{n-1}}\|A_N(a)\|$.

Both limiting and non-limiting properties of the extreme singular values of $A_N$ have
received considerable attention from researchers. Let us refer to the classical
works \cite{YBK88} and \cite{BY93} regarding
almost sure convergence of appropriately normalized singular values
when the coordinates of the underlying distributions are i.i.d.; as well as more recent works
\cite{R99,J01,LPRT05,
RV09,FS10,ALPT10,ALPT11,V12a,MP12,SV13,MP14,KM15,T_IJM,GLPT14,GLPT15,Y14,Y15,PY14,T_AiM,CT15}.
For a more comprehensive list of results, we refer to surveys \cite{RV10} and \cite{V12b}.
Let us remark that the question of approximating the covariance matrix for
{\it log-concave} distributions appeared in geometric functional analysis in connection
with the problem of computing the volume of a convex set given by a separation oracle
(see \cite{KLS95}). That question was considered, in particular, in \cite{B96,R99} and
was completely resolved in \cite{ALPT10,ALPT11}.

The purpose of this note is to establish approximation properties of the sample
covariance matrix under very mild assumptions on the distribution.
Fix for a moment any $p>2$ and $B\geq 1$.
Assume that $X$ is a centered $n$-dimensional random vector with a covariance matrix $\CovM$.
Assume that $X$ satisfies:
$$\Exp|\langle X,a\rangle|^{p}\leq B\bigl(\Exp\langle X,a\rangle^2\bigr)^{p/2}
=B\langle a,\CovM a\rangle^{p/2}\;\mbox{ for all }\;a\in\R^n.$$
The set of all distributions $F_X$ for centered random vectors $X$ satisfying the above condition
will be denoted by $\Class(n,p,B)$.
It is not difficult to check that the class $\Class(n,p,B)$ is closed under invertible linear transformations
(in the sense discussed above).
For isotropic distributions, the above condition is simplified to
$\Exp|\langle X,a\rangle|^p\leq B$ for all $a\in\Sph^{n-1}$.

The main result of the note is the following theorem:
\begin{theor}\label{t: main}
There is a non-increasing function $\nu:(2,\infty)\to\R_+$ with the following property:
Let $p>2$, $B\geq 1$, and assume that $N\geq 2n$.
Further, let $X$ be a centered $n$-dimensional
random vector with covariance matrix $\CovM$, whose distribution belongs to the class $\Class(n,p,B)$.
Let $X_1,X_2,\dots,X_N$ be independent copies of $X$, and let $\CovM_N=\frac{1}{N}\sum_{i=1}^N X_i{X_i}^T$.
Then the sample covariance matrix $\CovM_N$ satisfies
$$\nu(p)^{-1}\frac{\|\CovM_N-\CovM\|_{2\to 2}}{\|\CovM\|_{2\to 2}}\leq
\frac{1}{N}\max\limits_{i\leq N}\langle X_i,\CovM^{-1} X_i\rangle
+B^{2/p}\Bigl(\frac{n}{N}\Bigr)^{\frac{p-2}{p}}\log^4\frac{N}{n}+B^{2/p}
\Bigl(\frac{n}{N}\Bigr)^{\frac{\min(p,4)-2}{\min(p,4)}}
$$
with probability at least $1-\frac{1}{n}$.
\end{theor}

\bigskip

Note that the right-hand side of the above expression
depends on the precision matrix $\CovM^{-1}$.
As an additional assumption on the distribution,
one can make sure that
$\frac{1}{N}\max\limits_{i\leq N}\langle X_i,\CovM^{-1} X_i\rangle$
is typically smaller by the order of magnitude than the remaining summands.
Such an assumption implies that $X$ is concentrated
in the norm $\sqrt{\langle\cdot,\CovM^{-1}\cdot\rangle}$.
As an example, assuming that $\|\CovM^{-1}\|_{2\to 2},\|\CovM\|_{2\to 2}\leq C'$
and $\|X\|\leq C\sqrt{n}$ with very large probability for some constants $C,C'>0$, we
get $\langle X,\CovM^{-1} X\rangle\leq C'\|X\|^2\leq C'C^2n$
with high probability, so that the summand $\frac{1}{N}\max\limits_{i\leq N}\langle X_i,\CovM^{-1} X_i\rangle$
can be disregarded.

\bigskip

For $p\in(2,4]$, the last summand in the estimate of Theorem~\ref{t: main} is dominated by the second one,
and we can rewrite the conclusion
of the theorem as
$$\Prob\Bigl\{\nu(p)^{-1}\frac{\|\CovM_N-\CovM\|_{2\to 2}}{\|\CovM\|_{2\to 2}}\leq
\frac{1}{N}\max\limits_{i\leq N}\langle X_i,\CovM^{-1} X_i\rangle
+B^{2/p}\Bigl(\frac{n}{N}\Bigr)^{\frac{p-2}{p}}\log^4\frac{N}{n}\Bigr\}\geq 1-\frac{1}{n}.
$$
On the other hand, since $\log^4\frac{N}{n}$
grows with $N$ slower than any positive power of $\frac{N}{n}$,
for $p>4$ we can essentially disregard the second
summand in the estimate of Theorem~\ref{t: main}.
Let us provide a separate statement, which we formulate for isotropic distributions.
\begin{cor}\label{c: main}
There is a non-increasing function $\widetilde\nu:(4,\infty)\to\R_+$ with the following property:
Let $p>4$, $B\geq 1$, and assume that $N\geq 2n$.
Let $X$ be a centered isotropic vector with $\sup\limits_{a\in\Sph^{n-1}}\Exp|\langle X,a\rangle|^p\leq B$,
and let $X_1,X_2,\dots,X_N$ be its independent copies.
Then the sample covariance matrix $\CovM_N=\frac{1}{N}\sum_{i=1}^N X_i{X_i}^T$ satisfies
$$\Prob\Bigl\{\widetilde\nu(p)^{-1}\|\CovM_N-\Id_n\|_{2\to 2}\leq
\frac{1}{N}\max\limits_{i\leq N}\|X_i\|^2+B^{2/p}\sqrt{\frac{n}{N}}\Bigr\}\geq 1-\frac{1}{n}.$$
\end{cor}

\bigskip

In case when the coordinates of the random vector $X$ are i.i.d.\ centered
random variables with a bounded fourth moment, the well known
result of Z.D.~Bai and Y.Q.~Yin \cite{BY93} implies that
$$\|\CovM_N-\Id_n\|_{2\to 2}
= O\Bigl(\sqrt{\frac{n}{N}}\Bigr)$$
with probability close to one. In this connection,
Corollary~\ref{c: main} can be viewed as a Bai--Yin type estimate
for quite general class of distributions.

\bigskip

Let us make some further remarks.
For $p>2$, $B\geq 1$ and for all isotropic distributions from $\Class(n,p,B)$,
Theorem~\ref{t: main} provides the following bound for the extreme singular values of matrix $A_N$:
\begin{align*}
&N-K\max\limits_{i\leq N}\|X_i\|^2-KN\Bigl(\frac{n}{N}\Bigr)^{\frac{p-2}{p}}
\log^4\frac{N}{n}
\leq s_{\min}(A_N)^2\\
&\hspace{2cm}\leq s_{\max}(A_N)^2
\leq 1+K\max\limits_{i\leq N}\|X_i\|^2+KN\Bigl(\frac{n}{N}\Bigr)^{\frac{p-2}{p}}
\log^4\frac{N}{n},\;\;\;\mbox{if $p\leq 4$,}
\end{align*}
and
\begin{align*}
N-K\max\limits_{i\leq N}\|X_i\|^2
-K\sqrt{nN}
&\leq s_{\min}(A_N)^2\\
&\leq s_{\max}(A_N)^2
\leq N+K\max\limits_{i\leq N}\|X_i\|^2
+K\sqrt{nN},\;\;\;\mbox{if $p> 4$}
\end{align*}
with probability $\geq 1-\frac{1}{n}$,
where $K=K(p,B)>0$ depends only on $p,B$.
Note that better estimates for the {\it smallest} singular value were previously obtained
in \cite{KM15} and later strengthened in \cite{Y14,Y15}.
The papers \cite{SV13} and \cite{KM15} were apparently the first ones where lower bounds
for the smallest singular value were given
in quite a general setting without any restrictions on the magnitude of the
matrix norm $\|A_N\|_{2\to 2}$.
The novelty of our work consists in proving the upper bound for the largest singular value.
This problem has been extensively studied in the literature.
In \cite{ALPT10,ALPT11}, an analog of Theorem~\ref{t: main} was proved for distributions
with sub-exponential tails of one-dimensional projections.
In paper \cite{SV13}, just ``$2+\varepsilon$'' moment assumptions were employed, but,
as an additional requirement, the authors assumed certain tail decay for {\it all} projections (of any rank)
of the random vector.
In \cite{MP14}, an equivalent of Theorem~\ref{t: main} was proved under $8+\varepsilon$
moment assumption, and, finally, in \cite{GLPT15}, the result of \cite{MP14} was extended to $p>4$,
however the authors of \cite{GLPT15} did not obtain a Bai--Yin type estimate in the regime $4<p\leq 8$.

Thus, our input is two-fold: {\it first, we extend the theorems of \cite{MP14} and \cite{GLPT15}
to the range $p>2$, and, second, in the regime $p>4$ we obtain a Bai--Yin type estimate
for the largest singular value}. The factor ``$\log^4\frac{N}{n}$'' in the second summand of our bound,
which comes into play in the regime $2<p\leq 4$, seems excessive. We believe that
some essential new arguments are required to completely eliminate the log-factor, if it is at all possible.

\bigskip

As another illustration, let us consider a particular
form of the above theorem, which provides an estimate for the spectral norm of a square random matrix
with i.i.d.\ columns under very mild assumptions on the distribution:
\begin{theor}
Let $p>2$ and $B\geq 1$. Then there exist $K_1=K_1(p,B)$ and $K_2=K_2(p,B)$ depending only on
$p$ and $B$ with the following property:
Let $n\geq K_1$ and let $A$ be an $n\times n$ random matrix with i.i.d.\ columns $Y_1,Y_2,\dots,Y_n$,
where each $Y_i$ is a centered isotropic random vector satisfying
$\sup\limits_{a\in\Sph^{n-1}}\Exp|\langle Y_i,a\rangle|^p\leq B$.
Then the spectral norm of $A$ can be estimated as
$$\|A\|_{2\to 2}\leq K_2\max\limits_{i\leq N}\|Y_i\|$$
with probability at least $1-\frac{4}{n}$.
\end{theor}

The core of the proof of Theorem~\ref{t: main} is a ``chaining'' argument for quadratic forms already employed
in \cite{MP14,GLPT15}. At the same time, two crucial new ingredients are added:
First, we define a ``coloring'' of the sample, which is essentially
a truncation procedure for the inner products of the sample vectors.
Second is a Sparsifying Lemma, which allows to significantly decrease
cardinalities of $\varepsilon$-nets constructed in the proof,
thereby providing better probabilistic estimates for quadratic forms.
The Sparsifying Lemma allowed us to get the Bai--Yin type estimate for $p>4$, and
together with the coloring technique, to extend the range of admissible $p$'s to $(2,\infty)$.

The structure of the paper is the following: In Section~\ref{s: prelim},
we collect the notation and several auxiliary lemmas. In Section~\ref{s: coloring},
we define the coloring of the sample. In Sections~\ref{s: quadr det} and~\ref{s: quadr prob},
we define and estimate certain quadratic forms. In particular, the Sparsifying Lemma
(Lemma~\ref{l: s-r}) is given in Section~\ref{s: quadr det}.
Finally, in Section~\ref{s: proof}, we complete the proof of the main result.

\section{Preliminaries}\label{s: prelim}

The set of natural numbers will be denoted by $\N$, and reals --- by $\R$.
Given a natural number $k$, $[k]$ is the set $\{1,2,\dots,k\}$.
Cardinality of a finite set $S$ will be denoted by $|S|$.
For a real number $a$, $\lfloor a\rfloor$ is the largest integer not exceeding $a$, whereas
$\lceil a\rceil$ is the smallest integer greater or equal to $a$.
Let $\Sph^{N-1}$ be the standard unit sphere in $\R^N$ and
$\{e_i\}_{i=1}^N$ be the standard basis vectors in $\R^N$.
For brevity, for any subset $I\subset[N]$, by $\R^I$ we denote the span of the vectors $\{e_i\}_{i\in I}$.
Given a vector $y\in\R^N$, by $|y|\in\R_+^N$ we denote the vector of the absolute values of coordinates of $y$.

The standard inner product in $\R^N$ will be denoted by $\langle\cdot,\cdot\rangle$,
and the canonical Euclidean norm --- by $\|\cdot\|$.
For a vector $v\in\R^N$, $\|v\|_p$ ($1\leq p\leq\infty$) is the standard $\ell_p^N$ norm.
For a matrix $M$, its spectral norm is denoted by $\|M\|_{2\to 2}$.

Given a real non-negative sequence $(a_i)_{i=1}^N$, a subset $J\subset [N]$ and $k\in\N$, 
denote by $\kmax{k}{J}a_\ell$ the $k$-th largest element of the subsequence $(a_i)_{i\in J}$.
When $k>|J|$, we set $\kmax{k}{J}a_\ell:=0$.

Given a graph $G=(V,E)$, a vertex coloring of $G$ is an assignment of ``colors'' to all vertices such that
no adjacent vertices share the same color. The smallest possible number of colors sufficient
to assign a vertex coloring for $G$ is called {\it the chromatic number of $G$} and is denoted by $\chi(G)$.

For any $\rho>0$ and a subset $S\subset\R^N$, a {\it Euclidean $\rho$-net} $\Net$ in $S$
is any subset of $S$ such that for every $x\in S$ there is $y\in\Net$ with $\|x-y\|\leq \rho$.
If, additionally, one can always find $y\in\Net$ with $\supp y\subset \supp x$ and $\|x-y\|\leq\rho$
then we will call $\Net$ {\it a support-preserving} $\rho$-net.

A vector $y\in\R^N$ is {\it $r$-sparse} (for some $r\geq 0$) if $|\supp y|\leq r$.
The following lemma can be proved by standard arguments:
\begin{lemma}\label{l: s-p net}
For every $\rho\in(0,1]$ and any natural $r\leq N$, there exists a support-preserving
$\rho$-net $\Net$ in the set of all $r$-sparse unit vectors in $\R^N$
of cardinality at most $\bigl(\frac{C_{\text{\tiny\ref{l: s-p net}}}N}{\rho r}\bigr)^r$.
Here, $C_{\text{\tiny\ref{l: s-p net}}}>0$ is a universal constant.
\end{lemma}

The next lemma, stated in \cite{GLPT15} (the argument appeared already in \cite{ALPT11}), will be very helpful for us.
\begin{lemma}[{\cite[Lemma~4.1]{GLPT15}}]\label{l: quadratic net}
Let $M$ be an $n\times n$ matrix, $\rho\in(0,1/2)$, and let $\Net$ be a Euclidean $\rho$-net
in $\Sph^{n-1}$. Then
$$\sup\limits_{y\in \Sph^{n-1}}|\langle My,y\rangle|\leq (1-2\rho)^{-1}\sup\limits_{z\in\Net}|\langle Mz,z\rangle|.$$
\end{lemma}

Next, we recall two well known inequalities regarding the distribution of sums of independent random variables.
\begin{lemma}[W.Hoeffding, \cite{Hoeffding}]\label{l: hoeff}
Let $\xi_1,\xi_2,\dots,\xi_m$ be independent random variables, such that $\xi_i\in[a_i,b_i]$ a.s.\ for some
numbers $a_i,b_i\in\R$ ($i=1,2,\dots,m$). Then
$$\Prob\Bigl\{\sum\limits_{i=1}^m\xi_i-\sum\limits_{i=1}^m\Exp\xi_i\geq mt\Bigr\}
\leq\exp\Bigl(-2m^2t^2/\sum_{i=1}^m(b_i-a_i)^2\Bigr),\;\;t>0.$$
\end{lemma}

Given a random variable $\xi$, its {\it L\'evy concentration function} $\concf(\xi,\cdot)$
is defined as
$$\concf(\xi,t)=\sup\limits_{\lambda\in\R}\Prob\bigl\{|\xi-\lambda|\leq t\bigr\},\;\;t\geq 0.$$  

\begin{lemma}[H.Kesten, \cite{Kesten}]\label{l: kesten}
Let $\xi_1,\xi_2,\dots,\xi_m$ be independent random variables, and let $0<a_1,a_2,\dots,a_m\leq 2R$
be some real numbers. Then
$$\concf\Bigl(\sum\limits_{j=1}^m\xi_j,R\Bigr)\leq C_{\text{\tiny\ref{l: kesten}}}R
\frac{\sum\nolimits_{j=1}^m {a_j}^2\bigl(1-\concf(\xi_j,a_j)\bigr)\concf(\xi_j,R)}
{\bigl(\sum\nolimits_{j=1}^m {a_j}^2(1-\concf(\xi_j,a_j))\bigr)^{3/2}}.$$
Here, $C_{\text{\tiny\ref{l: kesten}}}>0$ is a universal constant.
\end{lemma}

The next lemma provides an elementary estimate of order statistics for a set of independent non-negative
variables.
\begin{lemma}\label{l: os est}
Let $h\geq 1$, $B\geq 1$, $r\in\N$ and let $\xi_1,\xi_2,\dots,\xi_r$
be independent non-negative random variables
such that $\Exp{\xi_i}^{h}\leq B$, $i=1,2,\dots,r$.
Then for any $m\leq r$ and $\tau>0$ we have
$$\Prob\bigl\{\kmax{m}{[r]}\xi_\ell\geq \tau\bigr\}\leq \biggl(\frac{eBr}{\tau^h m}\biggr)^m.$$
\end{lemma}
\begin{proof}
We have
\begin{align*}
\Prob\bigl\{\kmax{m}{[r]}\xi_\ell\geq \tau\bigr\}
\leq {r\choose m}\biggl(\frac{B}{\tau^h}\biggr)^m
\leq \biggl(\frac{eBr}{\tau^h m}\biggr)^m.
\end{align*}
\end{proof}

A centered random vector $X$ in $\R^n$ is {\it isotropic} if its covariance matrix $\Exp XX^T$ is the identity.
Let us give a simple bound for the norm of an isotropic vector assuming certain moment
conditions on its one-dimensional projections:
\begin{lemma}\label{l: vector lp norm}
Let $X$ be a centered $n$-dimensional isotropic vector, an suppose that for some $p>2$ and $B\geq 1$
we have
$$\sup\limits_{y\in\Sph^{n-1}}\Exp|\langle X,y\rangle|^p\leq B.$$
Then for any $\tau>0$ we have
$$\Prob\bigl\{\|X\|\geq \tau\bigr\}\leq Bn^{p/2}\tau^{-p}.$$
\end{lemma}
\begin{proof}
Note that $\|X\|\leq n^{1/2-1/p}\|X\|_p$ (deterministically), whence
$$\Exp\|X\|^{p}\leq n^{p/2-1}\Exp\|X\|_p^p\leq Bn^{p/2}.$$
Then, by Markov's inequality,
$$\Prob\bigl\{\|X\|\geq \tau\bigr\}\leq Bn^{p/2}\tau^{-p}.$$
\end{proof}

\section{Coloring the sample}\label{s: coloring}

Let $X_1,X_2,\dots,X_N$ be the i.i.d.\ copies of a centered $n$-dimensional isotropic vector $X$.
Further, fix a number $H>0$. We construct a random undirected graph $\Graph_H$ with the vertex set $[N]$
by defining its edge set as
$$\bigl\{(i,j):\,1\leq i< j\leq N,\,|\langle X_i,X_j\rangle|> H\max\limits_{h\leq N}\|X_h\|\bigr\}.$$

Let $\chi(\Graph_H)$ be the chromatic number of the graph. In what follows, for each $H>0$
we define a random partition $\{\ColorC^H_m\}_{m=1}^N$ of $[N]$, measurable with respect to the $\sigma$-algebra
generated by $X_1,X_2,\dots,X_N$, and satisfying the following two conditions:
\begin{enumerate}
\item[1)] $\ColorC^H_m=\emptyset$ for all $m>\chi(\Graph_H)$;
\item[2)] For any $m\leq N$ and $i,j\in \ColorC^H_m$ with $i\neq j$, the vertices $i$ and $j$ are not adjacent within $\Graph_H$, i.e.\
$|\langle X_i,X_j\rangle|\leq H\max\limits_{h\leq N}\|X_h\|$.
\end{enumerate}

The collection $\{\ColorC^H_m\}_{m=1}^N$ will be called {\it the coloring of the sample $X_1,X_2,\dots,X_N$ with
threshold $H$}.
Such a coloring will act as a way to ``truncate'' the inner products $\langle X_i,X_j\rangle$ and
will be employed when estimating quadratic forms in Sections~\ref{s: quadr det} and~\ref{s: quadr prob}.
At a more technical level, our estimate of the largest eigenvalue of the matrix
$\sum_{i=1}^N X_i{X_i}^T$ involves expressions
$\log (n)\,\max\limits_{i\neq j\in\ColorC}|\langle X_i,X_j\rangle|$
for some subsets $\ColorC\subset[N]$ (see Proposition~\ref{p: recurrent},
where they appear first time). A trivial upper bound
$\log (n)\,\max\limits_{i\neq j\in\ColorC}|\langle X_i,X_j\rangle|\leq \log(n)\max\limits_{i\leq N}\|X_i\|^2$
is not useful here; instead we build the argument in a way
that produces an upper bound of the form
$$\lambda_{\max}\Bigl(\sum_{i=1}^N X_i{X_i}^T\Bigr)
\lesssim\log (n)\,\sum\limits_{m=1}^N\max\limits_{i\neq j\in\ColorC^H_m}|\langle X_i,X_j\rangle|+\ldots
\leq\chi(\Graph_H)H\max\limits_{i\leq N}\|X_i\|+\ldots$$
Then a proper definition of $H$, together with a control of the random quantity $\chi(\Graph_H)$,
give a satisfactory estimate for $\lambda_{\max}$. In a sense, we partition the original sample $X_1,X_2,\dots,X_N$
into several subsets in such a way that within each subset the vectors are ``almost'' pairwise orthogonal.

The next statement provides tail bounds for the chromatic number $\chi(\Graph_H)$:

\begin{prop}\label{p: chrom}
Assume that for some $p>2$ and $B\geq 1$ we have
$\Exp|\langle X,y\rangle|^p\leq B$ for all $y\in\Sph^{n-1}$.
Then for any $H>0$ and any integer $m>1$ the chromatic number of $\Graph_H$ satisfies $\chi(\Graph_H)\leq m$ with probability
at least $1-\bigl(BNH^{-p}\bigr)^{m}n^{p/2}$.
\end{prop}
\begin{proof}
Let us introduce an auxiliary random process $Y(i)$ on $[N]$ with values in $\N$,
where $Y(1):=1$ (constant) and for all $i=2,3,\dots,N$:
$$
Y(i):=\min\bigl\{r\in\N:\, \forall j<i\,(j\in\N)\,\mbox{with $Y(j)=r$ we have $|\langle X_i,X_j\rangle|\leq H\|X_j\|$}\bigr\}.
$$
Note that by the very definition of $Y(i)$, we have that any two numbers $i\neq j\in[N]$ such that $Y(i)=Y(j)$,
are not adjacent in $\Graph_H$; in particular, $\chi(\Graph_H)\leq \max\limits_{i\in[N]}Y(i)$.
Next, for each $i>1$ and $m\geq 1$ we have
\begin{align*}
\Prob&\bigl\{Y(i)=m+1\bigr\}\\
&\leq\Prob\bigl\{\exists\,\ell\leq i-1\mbox{ such that $|\langle X_i,X_\ell\rangle|>H\|X_\ell\|$ and $Y(\ell)=m$}\bigr\}\\
&\leq\sum\limits_{\ell=1}^{i-1}\Prob\bigl\{|\langle X_i,X_\ell\rangle|>H\|X_\ell\|\mbox{ and }Y(\ell)=m\bigr\}\\
&\leq BH^{-p}\sum\limits_{\ell=1}^{i-1}\Prob\bigl\{Y(\ell)=m\bigr\}\\
&\leq BH^{-p}\Exp|\{j\leq N:\,Y(j)=m\}|.
\end{align*}
Hence,
$$
\Exp|\{j\leq N:\,Y(j)=m+1\}|
\leq BNH^{-p}\Exp|\{j\leq N:\,Y(j)=m\}|.
$$
Next, in view of Lemma~\ref{l: vector lp norm},
\begin{align*}
\Exp\bigl|\bigl\{j\leq N:\,Y(j)=2\bigr\}\bigr|
&\leq \Exp\bigl|\bigl\{j\leq N:\,\|X_j\|> H\bigr\}\bigr|\\
&=N\Prob\bigl\{\|X_1\|> H\bigr\}\\
&\leq BNH^{-p}n^{p/2}.
\end{align*}
Combining the estimates, we obtain for every $m\geq 1$:
$$\Exp|\{j\leq N:\,Y(j)=m+1\}|\leq \bigl(BNH^{-p}\bigr)^m n^{p/2}.$$
Note that the set of values $\{Y(j):\,j\leq N\}$ is an interval in $\N$, whence
\begin{align*}
\Prob\bigl\{\chi(\Graph_H)\geq m+1\bigr\}
&\leq\Prob\bigl\{\exists j\leq N\mbox{ with }Y(j)= m+1\bigr\}\\
&\leq\Exp|\{j\leq N:\,Y(j)=m+1\}|\\
&\leq \bigl(BNH^{-p}\bigr)^m n^{p/2}.
\end{align*}
\end{proof}

\section{Quadratic forms --- Deterministic estimates}\label{s: quadr det}

As before, let $X$ be a centered random vector in $\R^n$ and $X_1,X_2,\dots,X_N$ be its independent copies.
Additionally, we assume that the covariance matrix of $X$ is the identity.
By $A_N$ we denote the $N\times n$ random matrix with rows $X_1,X_2,\dots,X_N$.
For every natural $k\leq N$ and any subset $\ColorC\subset[N]$, denote
\begin{equation}\label{eq: f def}
f(k,\ColorC):=\sup\limits_{\substack{y\in\Sph^{N-1},\\
\supp y\subset \ColorC,\\ |\supp y|\leq k}}\Bigl\|\sum\limits_{i=1}^N y_i X_i\Bigr\|^2.
\end{equation}
Obviously, $f(N,[N])=\|A_N\|_{2\to 2}^2$.
Later, we will take $\ColorC$ to be one of classes from the coloring defined in the previous section,
in particular, $\ColorC$ will be a {\it random} set depending on $X_1,X_2,\dots,X_N$.
In this section, we will not estimate probabilities of any events, but instead produce
deterministic estimates for $f(k,\ColorC)$ as well as other quantities considered below.
The next relations provide a basis for our analysis. We have
\begin{align*}
f(k,\ColorC)&\leq\sup\limits_{\substack{y\in\Sph^{N-1},\\ \supp y\subset \ColorC,\\ |\supp y|\leq k}}\sum\limits_{i=1}^N y_i^2\|X_i\|^2
+
\sup\limits_{\substack{y\in\Sph^{N-1},\\ \supp y\subset \ColorC,\\ |\supp y|\leq k}}\sum\limits_{i\neq j}y_i y_j \langle X_i,X_j\rangle\\
&\leq\max\limits_{i\leq N}\|X_i\|^2+
\sup\limits_{\substack{y,z\in\Sph^{N-1},\\ \supp y,\supp z\subset \ColorC,\\ |\supp y|,
|\supp z|\leq k}}\sum\limits_{i\neq j}y_i z_j \langle X_i,X_j\rangle.
\end{align*}
Next, denoting $I^c:=[N]\setminus I$ for any $I\subset[N]$, we get:
\begin{align*}
\sup\limits_{\substack{y,z\in\Sph^{N-1},\\ \supp y,\supp z\subset \ColorC,\\ |\supp y|,
|\supp z|\leq k}}\sum\limits_{i\neq j}y_i z_j \langle X_i,X_j\rangle
&=
2^{-N+2}\sup\limits_{\substack{y\in\Sph^{N-1},\\ \supp y\subset \ColorC,\\ |\supp y|\leq k}}
\sup\limits_{\substack{z\in\Sph^{N-1},\\ \supp z\subset \ColorC,\\ |\supp z|\leq k}}
\sum\limits_{I\subset[N]}\bigl\langle \sum\limits_{i\in I} y_i X_i,\sum\limits_{j\in I^c} z_jX_j\bigr\rangle\\
&\leq
2^{-N+2}\sum\limits_{I\subset[N]}\sup\limits_{\substack{y\in\Sph^{N-1},\\ |\supp y|\leq k}}
\sup\limits_{\substack{z\in\Sph^{N-1},\\ |\supp z|\leq k}}
\bigl\langle \sum\limits_{i\in I\cap \ColorC} y_i X_i,\sum\limits_{j\in I^c\cap\ColorC} z_jX_j\bigr\rangle.
\end{align*}
For each $I\subset [N]$, denote
\begin{equation}\label{eq: g def}
g(k,\ColorC,I):=
\sup\limits_{\substack{y\in\Sph^{N-1},\\ |\supp y|\leq k}}
\sup\limits_{\substack{z\in\Sph^{N-1},\\ |\supp z|\leq k}}
\bigl\langle \sum\limits_{i\in I\cap\ColorC} y_i X_i,\sum\limits_{j\in I^c\cap\ColorC} z_jX_j\bigr\rangle.
\end{equation}
Further, for every vector $v\in\R^N$ and any $i\leq N$ we set
\begin{equation}\label{eq: w def}
W_{v,i}:=\langle X_i,\sum\limits_{j=1}^N v_j X_j\rangle.
\end{equation}
We recall that for any sequence $(a_\ell)_{\ell\in I^c\cap\ColorC}$
of non-negative real numbers,
by $\kmax{j}{I^c\cap\ColorC}a_\ell$
we denote the $j$-th largest element of the sequence.
Then we have for any integer $m\leq k$:
\begin{align*}
g(k,\ColorC,I)&=\sup\limits_{\substack{y\in\Sph^{N-1},\\ |\supp y|\leq k,\\ \supp y\subset I\cap\ColorC}}
\sup\limits_{\substack{z\in\Sph^{N-1},\\ |\supp z|\leq k,\\ \supp z\subset I^c\cap\ColorC}}\sum\limits_{j=1}^N
z_j W_{y,j}\\
&=\sup\limits_{\substack{y\in\Sph^{N-1},\\ |\supp y|\leq k,\\ \supp y\subset I\cap\ColorC}}
\Bigl(\sum\limits_{j=1}^k\kmax{j}{I^c\cap\ColorC}{W_{y,\ell}}^2\Bigr)^{1/2}\\
&\leq
\sup\limits_{\substack{y\in\Sph^{N-1},\\ |\supp y|\leq k,\\ \supp y\subset I\cap\ColorC}}
\Bigl(\sum\limits_{j=1}^m\kmax{j}{I^c\cap\ColorC}{W_{y,\ell}}^2\Bigr)^{1/2}
+\sup\limits_{\substack{y\in\Sph^{N-1},\\ |\supp y|\leq k,\\ \supp y\subset I\cap\ColorC}}
\Bigl(\sum\limits_{j=m+1}^k\kmax{j}{I^c\cap\ColorC}{W_{y,\ell}}^2\Bigr)^{1/2}\\
&\leq\sup\limits_{\substack{y\in\Sph^{N-1},\\ |\supp y|\leq k,\\ \supp y\subset I\cap\ColorC}}
\sup\limits_{\substack{z\in\Sph^{N-1},\\ |\supp z|\leq m,\\ \supp z\subset I^c\cap\ColorC}}
\bigl\langle \sum\limits_{i=1}^N y_i X_i,\sum\limits_{j=1}^N z_jX_j\bigr\rangle+
\sqrt{k}\sup\limits_{\substack{y\in\Sph^{N-1},\\ |\supp y|\leq k,\\ \supp y\subset I\cap\ColorC}}
\kmax{m}{I^c\cap\ColorC}{|W_{y,\ell}|}.
\end{align*}
Further,
\begin{align*}
&\sup\limits_{\substack{y\in\Sph^{N-1},\\ |\supp y|\leq k,\\ \supp y\subset I\cap\ColorC}}
\sup\limits_{\substack{z\in\Sph^{N-1},\\ |\supp z|\leq m,\\ \supp z\subset I^c\cap\ColorC}}
\bigl\langle \sum\limits_{i=1}^N y_i X_i,\sum\limits_{j=1}^N z_jX_j\bigr\rangle\\
&\hspace{1cm}=\sup\limits_{\substack{z\in\Sph^{N-1},\\ |\supp z|\leq m,\\ \supp z\subset I^c\cap\ColorC}}
\sup\limits_{\substack{y\in\Sph^{N-1},\\ |\supp y|\leq k,\\ \supp y\subset I\cap\ColorC}}
\sum\limits_{j=1}^N
y_j W_{z,j}\\
&\hspace{1cm}=\sup\limits_{\substack{z\in\Sph^{N-1},\\ |\supp z|\leq m,\\ \supp z\subset I^c\cap\ColorC}}
\Bigl(\sum\limits_{j=1}^k\kmax{j}{I\cap\ColorC}{W_{z,\ell}}^2\Bigr)^{1/2}\\
&\hspace{1cm}\leq\sup\limits_{\substack{z\in\Sph^{N-1},\\ |\supp z|\leq m,\\ \supp z\subset I^c\cap\ColorC}}
\sup\limits_{\substack{y\in\Sph^{N-1},\\ |\supp y|\leq m,\\ \supp y\subset I\cap\ColorC}}
\bigl\langle \sum\limits_{i=1}^N y_i X_i,\sum\limits_{j=1}^N z_jX_j\bigr\rangle+
\sqrt{k}\sup\limits_{\substack{z\in\Sph^{N-1},\\ |\supp z|\leq m,\\ \supp z\subset I^c\cap\ColorC}}
\kmax{m}{I\cap\ColorC}{|W_{z,\ell}|}.
\end{align*}
Thus, we can write for $m\leq k$:
\begin{align}
&g(k,\ColorC,I)\leq g(m,\ColorC,I)+\nonumber\\
&\hspace{1cm}\sqrt{k}\sup\limits_{\substack{y\in\Sph^{N-1},\\ |\supp y|\leq k,\\ \supp y\subset I\cap\ColorC}}
\kmax{m}{I^c\cap\ColorC}{|W_{y,\ell}|}+
\sqrt{k}\sup\limits_{\substack{z\in\Sph^{N-1},\\ |\supp z|\leq k,\\ \supp z\subset I^c\cap\ColorC}}
\kmax{m}{I\cap\ColorC}{|W_{z,\ell}|}.\label{eq: raw recurrent}
\end{align}
Let us remark that estimates for quadratic forms similar to the ones above, appeared
in literature before.
In particular, we refer to a work of J.~Bourgain \cite{B96}, which
deals with approximating covariance matrices
of log-concave distributions (see also \cite{ALPT10}), as well as
papers \cite{MP12,MP14} where a ``chaining'' argument
was employed for dealing with heavy-tailed distributions
(see \cite{GLPT15} for further development of the technique).

Unlike the above computations, the next lemma is a new addition to the arguments
employed in \cite{MP14,GLPT15}.
It provides a ``sparsifying'' technique
which will allow us to considerably decrease cardinalities of $\varepsilon$-nets
involved in the proof and, as a result,
weaken the moment assumptions on the distributions.
\begin{lemma}[Sparsifying Lemma]\label{l: s-r}
Let $\delta\in(0,1]$, $k\geq 12/\delta^2$, $m\geq 4$, and let $T=(t_{ij})$ be an $m\times k$ matrix of reals.
Then for any $y\in \Sph^{k-1}$ there is a coordinate projection $\Proj:\R^k\to\R^k$ of $\rank \Proj\leq \delta k$
such that
$${C_{\text{\tiny\ref{l: s-r}}}}^{-1}\delta^2\min\limits_{\ell\leq m}|Ty|_{\ell}
\leq \frac{\max_{i,j}|t_{ij}|}{\sqrt{k}}+\kmax{\lfloor m/4\rfloor}{[m]}|T\Proj (y)|_{\ell}.$$
Here, $C_{\text{\tiny\ref{l: s-r}}}>0$ is a universal constant, and
$\kmax{\lfloor m/4\rfloor}{[m]}|T\Proj (y)|_{\ell}$ is the $\lfloor m/4\rfloor$-th largest coordinate of the vector
$|T\Proj(y)|\in\R^m$.
\end{lemma}

Arguments similar in spirit to Lemma~\ref{l: s-r}, and based on Maurey's empirical method,
have been recently employed
to verify RIP properties of the Fourier matrices (see, in particular,
\cite{RV08,B14,HR}).
Let us note that the dependence on $\delta$ of the left-hand side of the bound in Lemma~\ref{l: s-r}
can probably be improved,
decreasing the power of the logarithmic factor in the estimate from Theorem~\ref{t: main};
however, will not eliminate it completely.

\begin{proof}[Proof of Lemma~\ref{l: s-r}]
Fix a vector $y\in\Sph^{k-1}$. Without loss of generality, we can assume that all coordinates of $y$ are non-negative,
and that $\min_{i\leq m}\bigl|\sum_{j=1}^k t_{ij}y_j\bigr|>0$. Denote
$$J:=\bigl\{j\leq k:\,y_j\geq 2/\sqrt{\delta k}\bigr\}.$$
It is easy to see that $|J|\leq \delta k/4$.
Consider two cases:
\begin{enumerate}

\item[1)] The set
$$I:=\Bigl\{i\leq m:\,\Bigl|\sum\limits_{j=1}^k t_{ij}y_j\Bigr|> 2\Bigl|\sum\limits_{j\in J}t_{ij}y_j\Bigr|\Bigr\}$$
has cardinality less than $m/2$.
Then, taking $\Proj$ to be the orthogonal projection
onto the span of $\{e_j\}_{j\in J}$, we get
$$\Bigl|\sum\limits_{j=1}^k t_{ij}y_j\Bigr|\leq 2\Bigl|\sum\limits_{j\in J}t_{ij}y_j\Bigr|=2|T\Proj(y)|_{i}$$
for all $i\in I^c:=[m]\setminus I$ (with $|I^c|\geq m/2$), implying the statement.

\item[2)] The set $I$
has cardinality at least $m/2$.
For brevity, let us denote $J^c:=[k]\setminus J$.
First, assume that for some $i_0\in I$ and $j_0\in J^c$ we have
$|t_{i_0j_0}y_{j_0}|\geq \frac{\delta\sqrt{\delta}}{8eC_{\text{\tiny\ref{l: kesten}}}}
\bigl|\sum\nolimits_{j\in J^c}t_{i_0j}y_j\bigr|$.
Then, in view of the definition of $J$, we get
$$\max\limits_{i,j}|t_{ij}|\geq
|t_{i_0j_0}|\geq\frac{\delta^2 \sqrt{k}}{16eC_{\text{\tiny\ref{l: kesten}}}}\Bigl|\sum\limits_{j\in J^c}t_{i_0j}y_j\Bigr|
\geq\frac{\delta^2 \sqrt{k}}{32eC_{\text{\tiny\ref{l: kesten}}}}\Bigl|\sum\limits_{j=1}^k t_{i_0j}y_j\Bigr|
\geq \frac{\delta^2 \sqrt{k}}{32eC_{\text{\tiny\ref{l: kesten}}}}\min\limits_{i\leq m}\Bigl|\sum\limits_{j=1}^k t_{ij}y_j\Bigr|,$$
implying the statement.
For the rest of the proof, we will suppose that
\begin{equation}\label{eq: aux12}
|t_{iq}y_q|< \frac{\delta\sqrt{\delta}}{8eC_{\text{\tiny\ref{l: kesten}}}}\bigl|\sum\limits_{j\in J^c}t_{ij}y_j\bigr|\;\;
\mbox{for all $q\in J^c$ and $i\in I$}.
\end{equation}

Define a random coordinate projection $\widetilde\Proj:\R^k\to\R^k$ as follows:
Let $\{\eta_j\}_{j\in J^c}$ be i.i.d.\ Bernoulli ($0\text{-}1$) random variables with probability of success $\delta/2$,
and set $\Im\widetilde \Proj:=\spn\{\eta_j e_j\}_{j\in J^c}$.
Clearly, $\rank\widetilde\Proj=\sum_{j\in J^c}\eta_j$, and by Hoeffding's inequality (Lemma~\ref{l: hoeff})
we have $\Prob\{\rank\widetilde \Proj> \delta k\}\leq \exp(-\delta^2 k/4)\leq 0.1$.
We will show that for any index $i\in I$ we have
$$|T\widetilde\Proj(y)|_i=\Bigl|\sum\limits_{j\in J^c} \eta_j t_{ij}y_j\Bigr|\geq \frac{\delta}{4}\Bigl|\sum\limits_{j\in J^c}t_{ij}y_j\Bigr|
\geq \frac{\delta}{8}\Bigl|\sum\limits_{j=1}^k t_{ij}y_j\Bigr|$$
with probability at least $1-\exp(-1)$.
Fix any $i\in I$.
First, assume that
$$\bigl(\sum\limits_{j\in J^c}t_{ij}y_j\bigr)^2
\geq\frac{8}{\delta^2}\sum\limits_{j\in J^c}{t_{ij}}^2{y_j}^2.$$
Then by Hoeffding's inequality (Lemma~\ref{l: hoeff}), we have
\begin{align*}
\Prob\Bigl\{\Bigl|\sum_{j\in J^c}\eta_jt_{ij}y_j\Bigr|<\frac{\delta}{4}\Bigl|\sum\limits_{j\in J^c}t_{ij}y_j\Bigr|\Bigr\}
&\leq \exp\Bigl(-\frac{\delta^2}{8}\bigl(\sum\limits_{j\in J^c}t_{ij}y_j\bigr)^2/
\sum\limits_{j\in J^c}{t_{ij}}^2{y_j}^2\Bigr)\\
&\leq\exp(-1).
\end{align*}
Now, assume that $\bigl(\sum\nolimits_{j\in J^c}t_{ij}y_j\bigr)^2
<\frac{8}{\delta^2}\sum\nolimits_{j\in J^c}{t_{ij}}^2{y_j}^2$.
Then, applying Kesten's inequality (Lemma~\ref{l: kesten}) with
$R:=\frac{\delta\sqrt{\delta}}{8eC_{\text{\tiny\ref{l: kesten}}}}\bigl|\sum_{j\in J^c}t_{ij}y_j\bigr|$
and $a_j:=\frac{1}{2}|t_{ij}y_j|$ (note that $2R\geq a_j$ for all $j\in J^c$ in view of \eqref{eq: aux12}),
we obtain
\begin{align*}
\concf\Bigl(\sum_{j\in J^c}\eta_j t_{ij}y_j,R\Bigr)
&\leq C_{\text{\tiny\ref{l: kesten}}}R\Bigl(\sum\limits_{j\in J^c}{a_j}^2\bigl(1-\concf(\eta_j t_{ij} y_j,a_j)\bigr)\Bigr)^{-1/2}\\
&\leq \sqrt{\frac{8}{\delta}}C_{\text{\tiny\ref{l: kesten}}}R\Bigl(\sum\limits_{j\in J^c}{t_{ij}}^2{y_j}^2\Bigr)^{-1/2}\\
&\leq\frac{8C_{\text{\tiny\ref{l: kesten}}}R}{\delta\sqrt{\delta}}\Bigl|\sum\limits_{j\in J^c}t_{ij}y_j\Bigr|^{-1}\\
&\leq\exp(-1).
\end{align*}

Thus, for any $i\in I$ we have
$$\Prob\Bigl\{\Bigl|\sum_{j\in J^c}\eta_j t_{ij}y_j\Bigr|<
\frac{\delta\sqrt{\delta}}{8eC_{\text{\tiny\ref{l: kesten}}}}\Bigl|\sum_{j\in J^c}t_{ij}y_j\Bigr|\Bigr\}\leq\exp(-1),$$
whence, by the definition of $I$,
$$\Prob\Bigl\{\Bigl|\sum_{j\in J^c}\eta_j t_{ij}y_j\Bigr|\geq
\frac{\delta\sqrt{\delta}}{16eC_{\text{\tiny\ref{l: kesten}}}}\Bigl|\sum_{j=1}^k t_{ij}y_j\Bigr|\;\mbox{ and }\;
\rank\widetilde\Proj\leq\delta k\Bigr\}>\frac{1}{2},\;\;i\in I$$
(recall that $\Prob\{\rank\widetilde\Proj\leq\delta k\}\geq 0.9$).
This immediately implies that there is a (non-random) realization $\Proj$ of $\widetilde\Proj$ such that $\rank\Proj\leq \delta k$,
and
$$|T\Proj(y)|_i\geq
\frac{\delta\sqrt{\delta}}{16eC_{\text{\tiny\ref{l: kesten}}}}\Bigl|\sum_{j=1}^k t_{ij}y_j\Bigr|$$
for at least half of the indices $i\in I$, i.e.\ for at least $m/4$ indices. The result follows.
\end{enumerate}
\end{proof}

The next statement is an application of Lemma~\ref{l: s-r} to relation \eqref{eq: raw recurrent}:
\begin{lemma}\label{l: interm recurrent}
Let $\delta\in(0,1]$, $k\geq 12/\delta^2$ and $4\leq m\leq k$.
Further, let $I,\ColorC\subset[N]$ be subsets of $[N]$ (whether fixed or random). Then
\begin{align*}
&g(k,\ColorC,I)\leq g(m,\ColorC,I)+
2C_{\text{\tiny\ref{l: s-r}}}\delta^{-2}\max\limits_{i\neq j\in \ColorC}|\langle X_i,X_j\rangle|\\
&\hspace{0.5cm}+C_{\text{\tiny\ref{l: s-r}}}\sqrt{k}\delta^{-2}
\sup\limits_{\substack{y\in\Sph^{N-1},\\ |\supp y|\leq \delta k,\\ \supp y\subset I\cap\ColorC}}
\kmax{\lfloor m/4\rfloor}{I^c\cap\ColorC}{|W_{y,\ell}|}+
C_{\text{\tiny\ref{l: s-r}}}\sqrt{k}\delta^{-2}
\sup\limits_{\substack{z\in\Sph^{N-1},\\ |\supp z|\leq \delta k,\\ \supp z\subset I^c\cap\ColorC}}
\kmax{\lfloor m/4\rfloor}{I\cap\ColorC}{|W_{z,\ell}|},
\end{align*}
where $g$ and $W$ are defined in~\eqref{eq: g def} and~\eqref{eq: w def}.
\end{lemma}
\begin{proof}
Fix a realization of the vectors $X_1,X_2,\dots,X_N$ and of the sets $I,\ColorC$, and consider the quantity
$$\sup\limits_{\substack{y\in\Sph^{N-1},\\ |\supp y|\leq k,\\ \supp y\subset I\cap\ColorC}}
\kmax{m}{I^c\cap\ColorC}{|W_{y,\ell}|}.$$
Without loss of generality, we can assume that it is non-zero.
Let $\widetilde y\in\Sph^{N-1}$ with $|\supp \widetilde y|\leq k$ and $\supp \widetilde y\subset I\cap\ColorC$
be a vector which delivers the supremum in the above expression.
Note that necessarily $|I^c\cap\ColorC|\geq m$.
Let $U\subset I^c\cap\ColorC$
be a set of indices $\ell$ of cardinality $m$ corresponding to $m$ largest elements of the
sequence $(|W_{\widetilde y,\ell}|)_{\ell\in I^c\cap\ColorC}$, and let $V:=\supp\widetilde y\subset I\cap\ColorC$.
Then we define an $m\times |V|$ matrix $T=(t_{\ell j})$ whose elements are the inner products
$\langle X_\ell,X_j\rangle$ ($\ell\in U$, $j\in V$).
For convenience, we index the elements of the
matrix over the Cartesian product $U\times V$. Then we can define the multiplication $T\widetilde y$
in a natural way by setting $T\widetilde y:=(W_{\widetilde y,\ell})_{\ell\in U}$.
Note that
$$\min\limits_{\ell\in U}\Bigl|\sum\limits_{j\in V} t_{\ell j}\widetilde y_j\Bigr|=\kmax{m}{I^c\cap\ColorC}{|W_{\widetilde y,\ell}|}.$$
Then, applying Lemma~\ref{l: s-r}, we get that there is a coordinate projection $\Proj:\R^V\to\R^V$
of rank at most $\delta k$ such that
\begin{align*}
{C_{\text{\tiny\ref{l: s-r}}}}^{-1}\delta^2\kmax{m}{I^c\cap\ColorC}{|W_{\widetilde y,\ell}|}
&\leq
\frac{\max_{\ell,j}|t_{\ell j}|}{\sqrt{k}}+\kmax{\lfloor m/4\rfloor}{U}|T\Proj (\widetilde y)|_{\ell}\\
&\leq \frac{\max\limits_{i\neq j\in \ColorC}|\langle X_i,X_j\rangle|}{\sqrt{k}}
+\kmax{\lfloor m/4\rfloor}{I^c\cap\ColorC}|W_{\Proj \widetilde y,\ell}|.
\end{align*}
Hence,
$$\sup\limits_{\substack{y\in\Sph^{N-1},\\ |\supp y|\leq k,\\ \supp y\subset I\cap\ColorC}}
\kmax{m}{I^c\cap\ColorC}{|W_{y,\ell}|}
\leq \frac{C_{\text{\tiny\ref{l: s-r}}}\max\limits_{i\neq j\in \ColorC}|\langle X_i,X_j\rangle|}{\delta^2\sqrt{k}}+
C_{\text{\tiny\ref{l: s-r}}}\delta^{-2}\sup\limits_{\substack{y\in\Sph^{N-1},\\ |\supp y|\leq \delta k,\\ \supp y\subset I\cap\ColorC}}
\kmax{\lfloor m/4\rfloor}{I^c\cap\ColorC}{|W_{y,\ell}|}.$$
Repeating the argument for
$$\sup\limits_{\substack{z\in\Sph^{N-1},\\ |\supp z|\leq k,\\ \supp z\subset I^c\cap\ColorC}}
\kmax{m}{I\cap\ColorC}{|W_{z,\ell}|},$$
and applying relation \eqref{eq: raw recurrent}, we obtain the statement.
\end{proof}

The next lemma is a variation of the standard procedure of passing from supremum over a set of vectors
to the supremum over a net.

\begin{lemma}\label{l: net}
Let $\rho\in(0,1]$, $r,h,p,q\in\N$ with $r\geq 2$, and let $\Net$ be
a support-preserving Euclidean $\rho$-net on the set of all $h$-sparse unit vectors in $\R^q$. Further,
let $T=(t_{ij})$ be any $p\times q$ matrix. Then
$$\sup\limits_{\substack{u\in\Sph^{q-1},\\ |\supp u|\leq h}}\kmax{r}{[p]}|Tu|_\ell
\leq 2\sup\limits_{v\in\Net}\kmax{\lfloor r/2\rfloor}{[p]}|Tv|_{\ell}+\frac{4\rho}{\sqrt{r}}
\sup\limits_{\substack{u\in\Sph^{q-1},\\ |\supp u|\leq h}}\Bigl(\sum\limits_{i=1}^r \kmax{i}{[p]}{|Tu|_\ell}^2\Bigr)^{1/2}.$$
\end{lemma}
\begin{proof}
Without loss of generality, $r\leq p$.
Fix a vector $\widetilde u\in \Sph^{q-1}$ with $|\supp \widetilde u|\leq h$.
By the definition of $\Net$, there is $\widetilde v\in\Net$
with $\supp\widetilde v\subset\supp\widetilde u$ and $\|\widetilde u-\widetilde v\|\leq \rho$.
Assume that $2\kmax{\lfloor r/2\rfloor}{[p]}|T\widetilde v|_{\ell}<\kmax{r}{[p]}|T\widetilde u|_\ell$.
Let $\sigma$ be a permutation on $p$ elements such that the sequence $(|T \widetilde u|_{\sigma(i)})$, $1\leq i\leq p$,
is non-increasing. Then the last condition implies that there is a subset $J\subset [r]$ of cardinality at least $r/2$
such that $|T\widetilde u|_{\sigma(i)}>2|T\widetilde v|_{\sigma(i)}$ for all $i\in J$, implying that
$|T(\widetilde u-\widetilde v)|_i\geq \frac{1}{2}\kmax{r}{[p]}|T\widetilde u|_\ell$ for at least $r/2$ indices $i\in[p]$.
At the same time, $\|\widetilde u-\widetilde v\|\leq \rho$ and $|\supp (\widetilde u-\widetilde v)|\leq h$.
Setting $s:=\frac{\widetilde u-\widetilde v}{\|\widetilde u-\widetilde v\|}$, it follows that
$$\sup\limits_{\substack{u\in\Sph^{q-1},\\ |\supp u|\leq h}}\Bigl(\sum\limits_{i=1}^r \kmax{i}{[p]}{|Tu|_\ell}^2\Bigr)^{1/2}
\geq \Bigl(\sum\limits_{i=1}^r \kmax{i}{[p]}{|Ts|_\ell}^2\Bigr)^{1/2}
\geq \frac{1}{2\rho}\sqrt{\frac{r}{2}}\kmax{r}{[p]}|T\widetilde u|_\ell.$$
Thus, we have shown that for any vector $\widetilde u\in \Sph^{q-1}$ with $|\supp \widetilde u|\leq h$ we have
$$\kmax{r}{[p]}|T\widetilde u|_\ell
\leq 2\sup\limits_{v\in\Net}\kmax{\lfloor r/2\rfloor}{[p]}|Tv|_{\ell}+\frac{4\rho}{\sqrt{r}}
\sup\limits_{\substack{u\in\Sph^{q-1},\\ |\supp u|\leq h}}\Bigl(\sum\limits_{i=1}^r \kmax{i}{[p]}{|Tu|_\ell}^2\Bigr)^{1/2}.$$
Taking the supremum over all admissible $\widetilde u$, we get the result.
\end{proof}

Combining Lemmas~\ref{l: interm recurrent} and~\ref{l: net}, we obtain the main result of the section:
\begin{prop}\label{p: recurrent}
Let $I\subset[N]$ be a fixed subset; $\ColorC\subset[N]$ be random, and let $\delta\in(0,1/3)$, with $k\geq 24/\delta^2$
and $N\geq 128C_{\text{\tiny\ref{l: s-r}}}\delta^{-2}k$.
Denote $t:=\lfloor\log_2\frac{\delta^2 k}{24}\rfloor$
and define $k_j:=\lfloor k/2^j\rfloor$, $0\leq j\leq t$.
Then there are subsets of $\delta k_j$-sparse unit vectors $\Net_j$ and $\Net_j'$ ($0\leq j\leq t-1$)
supported on $I$ and $I^c$, respectively, such that
$1)$ $|\Net_j|,|\Net_j'|\leq \bigl(\frac{C_{\text{\tiny\ref{l: s-p net}}}N}{\delta k_j}\bigr)^{2\delta k_j}$
for all admissible $j$, and $2)$ we have
\begin{align*}
g(k,\ColorC,I)&\leq
C_{\text{\tiny\ref{p: recurrent}}}\delta^{-2}\log (k)\,\max\limits_{i\neq j\in \ColorC}|\langle X_i,X_j\rangle|\\
&\hspace{1cm}+C_{\text{\tiny\ref{p: recurrent}}}\delta^{-2}
\sum\limits_{j=0}^{t-1}\sqrt{k_j}\sup\limits_{u\in\Net_j}\kmax{\lfloor k_{j+1}/16\rfloor}{I^c}{|W_{u,\ell}|}\\
&\hspace{1cm}+C_{\text{\tiny\ref{p: recurrent}}}\delta^{-2}\sum\limits_{j=0}^{t-1}\sqrt{k_j}
\sup\limits_{v\in\Net_j'}\kmax{\lfloor k_{j+1}/16\rfloor}{I}{|W_{v,\ell}|}.
\end{align*}
Here, $C_{\text{\tiny\ref{p: recurrent}}}>0$ is a universal constant, and $g$ and $W$ are defined
by~\eqref{eq: g def} and~\eqref{eq: w def}.
\end{prop}
\begin{proof}
First, we fix any $j<t$ and consider the quantity
$g(k_j,\ColorC,I)$. We define $\Net_j\subset\R^{I}$ as a support-preserving $\frac{k_j}{N}$-net
in the set of $\delta k_j$-sparse unit vectors in $\R^{I}$,
of cardinality at most
$\bigl(\frac{C_{\text{\tiny\ref{l: s-p net}}}}{\delta}\bigr)^{\delta k_j}\bigl(\frac{N}{k_j}\bigr)^{2\delta k_j}
\leq \bigl(\frac{C_{\text{\tiny\ref{l: s-p net}}}N}{\delta k_j}\bigr)^{2\delta k_j}$
(such a net exists in view of Lemma~\ref{l: s-p net}).
Similarly, we let $\Net_j'\subset\R^{I^c}$ be a support-preserving $\frac{k_j}{N}$-net
in the set of $\delta k_j$-sparse unit vectors in $\R^{I^c}$, with $|\Net_j'|\leq
\bigl(\frac{C_{\text{\tiny\ref{l: s-p net}}}N}{\delta k_j}\bigr)^{2\delta k_j}$.
Now, in view of Lemma~\ref{l: interm recurrent}, we have
\begin{align*}
g(k_j,\ColorC,I)\leq g(k_{j+1},\ColorC,I)&+
2C_{\text{\tiny\ref{l: s-r}}}\delta^{-2}\max\limits_{i\neq j\in \ColorC}|\langle X_i,X_j\rangle|\\
&+C_{\text{\tiny\ref{l: s-r}}}\sqrt{k_j}\delta^{-2}\sup\limits_{\substack{y\in\Sph^{N-1},
\\ |\supp y|\leq \delta k_j,\\ \supp y\subset I\cap\ColorC}}
\kmax{\lfloor k_{j+1}/4\rfloor}{I^c\cap\ColorC}{|W_{y,\ell}|}\\
&+C_{\text{\tiny\ref{l: s-r}}}\sqrt{k_j}\delta^{-2}\sup\limits_{\substack{z\in\Sph^{N-1},
\\ |\supp z|\leq \delta k_j,\\ \supp z\subset I^c\cap\ColorC}}
\kmax{\lfloor k_{j+1}/4\rfloor}{I\cap\ColorC}{|W_{z,\ell}|}.
\end{align*}
Applying Lemma~\ref{l: net} with $r:=\lfloor k_{j+1}/4\rfloor$, $\rho:=\frac{k_j}{N}$, $h:=\lfloor\delta k_j\rfloor$
and a $|I^c\cap \ColorC|\times |I\cap\ColorC|$ matrix $T=(\langle X_i,X_j\rangle)$
($(i,j)\in (I^c\cap \ColorC)\times (I\cap\ColorC)$), and using the definition of $W$'s \eqref{eq: w def}, we get
\begin{align*}
&\sup\limits_{\substack{y\in\Sph^{N-1},\\ |\supp y|\leq \delta k_j,\\ \supp y\subset I\cap\ColorC}}
\kmax{\lfloor k_{j+1}/4\rfloor}{I^c\cap\ColorC}{|W_{y,\ell}|}\\
&\hspace{0.5cm}\leq 2\sup\limits_{u\in\Net_j}\kmax{\lfloor\lfloor k_{j+1}/4\rfloor/2\rfloor}{I^c}{|W_{u,\ell}|}
+\frac{4\frac{k_j}{N}}{\sqrt{\lfloor k_{j+1}/4\rfloor}}
\sup\limits_{\substack{y\in\Sph^{N-1},\\ |\supp y|\leq \delta k_j,\\ \supp y\subset I\cap\ColorC}}
\Bigl(\sum\limits_{i=1}^{\lfloor k_{j+1}/4\rfloor} \kmax{i}{I^c\cap\ColorC}{W_{y,\ell}}^2\Bigr)^{1/2}\\
&\hspace{0.5cm}\leq 2\sup\limits_{u\in\Net_j}\kmax{\lfloor\lfloor k_{j+1}/4\rfloor/2\rfloor}{I^c}{|W_{u,\ell}|}
+\frac{4\frac{k_j}{N}}{\sqrt{\lfloor k_{j+1}/4\rfloor}}g(k_{j+1},\ColorC,I).
\end{align*}
Carrying out analogous estimate for
$$\sup\limits_{\substack{z\in\Sph^{N-1},
\\ |\supp z|\leq \delta k_j,\\ \supp z\subset I^c\cap\ColorC}}
\kmax{\lfloor k_{j+1}/4\rfloor}{I\cap\ColorC}{|W_{z,\ell}|},$$
we obtain
\begin{align*}
g(k_j,\ColorC,I)\leq \Bigl(1+
\frac{32C_{\text{\tiny\ref{l: s-r}}}k_j}{\delta^2N}
\Bigr)g(k_{j+1},\ColorC,I)&+
2C_{\text{\tiny\ref{l: s-r}}}\delta^{-2}\max\limits_{i\neq j\in \ColorC}|\langle X_i,X_j\rangle|\\
&+2C_{\text{\tiny\ref{l: s-r}}}\sqrt{k_j}\delta^{-2}\sup\limits_{u\in\Net_j}\kmax{\lfloor k_{j+1}/16\rfloor}{I^c}{|W_{u,\ell}|}\\
&+2C_{\text{\tiny\ref{l: s-r}}}\sqrt{k_j}\delta^{-2}\sup\limits_{v\in\Net_j'}
\kmax{\lfloor k_{j+1}/16\rfloor}{I}{|W_{v,\ell}|}.
\end{align*}
Note that, by the restrictions on $N$,
$$\prod\limits_{j=0}^t\Bigl(1+\frac{32C_{\text{\tiny\ref{l: s-r}}}k_j}{\delta^2N}\Bigr)
\leq \exp\Bigl(\sum_{j=0}^t\frac{32C_{\text{\tiny\ref{l: s-r}}}k_j}{\delta^2N}\Bigr)
\leq \exp\Bigl(\frac{128C_{\text{\tiny\ref{l: s-r}}}k}{\delta^{2}N}\Bigr)\leq \exp(1).$$
Hence, recursively applying the above estimate for $g(k_j,\ColorC,I)$ for all $0\leq j<t$, we obtain
\begin{align*}
\exp(-1)g(k,\ColorC,I)\leq
g(k_t,\ColorC,I)
&+2C_{\text{\tiny\ref{l: s-r}}}\delta^{-2}\log_2 (k)\,\max\limits_{i\neq j\in \ColorC}|\langle X_i,X_j\rangle|\\
&+2C_{\text{\tiny\ref{l: s-r}}}\delta^{-2}\sum\limits_{j=0}^{t-1}
\sqrt{k_j}\sup\limits_{u\in\Net_j}\kmax{\lfloor k_{j+1}/16\rfloor}{I^c}{|W_{u,\ell}|}\\
&+2C_{\text{\tiny\ref{l: s-r}}}\delta^{-2}\sum\limits_{j=0}^{t-1}\sqrt{k_j}
\sup\limits_{v\in\Net_j'}\kmax{\lfloor k_{j+1}/16\rfloor}{I}{|W_{v,\ell}|}.
\end{align*}
It remains to note that the quantity $g(k_t,\ColorC,I)$ can be estimated as
\begin{align*}
g(k_t,\ColorC,I)
&\leq\sup\limits_{\substack{y\in\Sph^{N-1},\\ |\supp y|\leq k_t,\\ \supp y\subset I\cap\ColorC}}
\sup\limits_{\substack{z\in\Sph^{N-1},\\ |\supp z|\leq k_t,\\ \supp z\subset I^c\cap\ColorC}}
\sum\limits_{i=1}^N\sum\limits_{j=1}^N |y_i z_j\langle X_i,X_j\rangle|\\
&\leq \max\limits_{i\neq j\in \ColorC}|\langle X_i,X_j\rangle|
\sup\limits_{\substack{y\in\Sph^{N-1},\\ |\supp y|\leq k_t}}
\sup\limits_{\substack{z\in\Sph^{N-1},\\ |\supp z|\leq k_t}}\sum\limits_{i,j=1}^N|y_iz_j|\\
&= k_t\max\limits_{i\neq j\in \ColorC}|\langle X_i,X_j\rangle|\\
&\leq 48\delta^{-2}\max\limits_{i\neq j\in \ColorC}|\langle X_i,X_j\rangle|.
\end{align*}
\end{proof}

\section{Quadratic forms --- Probabilistic estimates}\label{s: quadr prob}

In this section, we apply the deterministic bounds from Section~\ref{s: quadr det} to
obtain estimates for the tail distribution of quantity $f(n,[N])$ defined in \eqref{eq: f def}.
We always assume that $X$ is an $n$-dimensional centered isotropic vector; $X_1,X_2,\dots,X_N$ are its independent copies,
and additionally suppose that $\sup\limits_{y\in \Sph^{n-1}}\Exp|\langle X,y\rangle|^p\leq B$ for some $p>2$ and $B\geq 1$.

Let us start with the following corollary of Proposition~\ref{p: recurrent}:
\begin{prop}\label{p: prob w}
There is a sufficiently large universal constant $C_{\text{\tiny\ref{p: prob w}}}$ with the following property:
Let $I\subset[N]$ be fixed and $\ColorC\subset[N]$ be random,
and let $n,N>1$ with $\log\frac{N}{n}\geq C_{\text{\tiny\ref{p: prob w}}}\max\bigl(1,1/(p-2)\bigr)$.
Then we have
\begin{align*}
g(n,\ColorC,I)\leq
C_{\text{\tiny\ref{p: prob w}}}\log^2\frac{N}{n}\,\log (n)\,\max\limits_{i\neq j\in \ColorC}|\langle X_i,X_j\rangle|
+\frac{C_{\text{\tiny\ref{p: prob w}}}pB^{1/p}}{p-2}\log^2\frac{N}{n}\,\sqrt{n}\Bigl(\frac{N}{n}\Bigr)^{1/p}\sqrt{f(n,[N])}
\end{align*}
with probability at least $1-\frac{1}{N^3}$.
\end{prop}
\begin{proof}
First, consider the case when $n< 24\log^2\frac{N}{n}$.
Then a crude deterministic bound on $g(n,\ColorC,I)$ gives
$$g(n,\ColorC,I)\leq n\max\limits_{i\neq j\in \ColorC}|\langle X_i,X_j\rangle|<24\log^2\frac{N}{n}
\,\max\limits_{i\neq j\in \ColorC}|\langle X_i,X_j\rangle|,$$
and we get the statement.

For the rest of the proof, we assume that $n\geq 24\log^2\frac{N}{n}$.
Define $\delta:=\frac{1}{\log(N/n)}$, $t:=\bigl\lfloor\log_2\frac{\delta^2 n}{24}\bigr\rfloor$
and $k_j:=\bigl\lfloor n/2^j\bigr\rfloor$ ($j=0,1,\dots,t$). We can assume that $\log(N/n)$ is sufficiently large, so that
the conditions of Proposition~\ref{p: recurrent} are satisfied.
Fix for a moment any $0\leq j<t$ and consider the quantity
$$\sup\limits_{u\in\Net_j}\kmax{\lfloor k_{j+1}/16\rfloor}{I^c}{|W_{u,\ell}|},$$
where $\Net_j\subset\R^{I}$ is defined in Proposition~\ref{p: recurrent}.
Fix any $u\in\Net_j$. Note that, conditioned on a realization of vectors $X_i$ ($i\in I$),
the quantities $W_{u,\ell}=\langle X_\ell,\sum_{i\in I} u_i X_i\rangle$ ($\ell\in I^c$) are jointly independent.
Moreover, in view of the moment assumptions on $X$,
the conditional expectation of $|W_{u,\ell}|^p$ given $X_i$ ($i\in I$), satisfies
$$\Exp\bigl(|W_{u,\ell}|^p\,|\,X_i,\;i\in I\bigr)\leq B\Bigl\|\sum_{i\in I} u_i X_i\Bigr\|^p
\leq B\bigl(f(n,[N])\bigr)^{p/2}.$$
Applying Lemma~\ref{l: os est} to $|W_{u,\ell}|$'s with
$\tau_j:=(32eB)^{1/p}\sqrt{f(n,[N])}\bigl(\frac{N}{k_{j+1}}\bigr)^{p^{-1}(1+256\delta)}$,
we get
\begin{align*}
\Prob\bigl\{\kmax{\lfloor k_{j+1}/16\rfloor}{I^c}{|W_{u,\ell}|}\geq \tau_j\bigr\}
&\leq \biggl(\frac{eBf(n,[N])^{p/2}N}{{\tau_j}^p \lfloor k_{j+1}/16\rfloor}\biggr)^{\lfloor k_{j+1}/16\rfloor}\\
&\leq \biggl(\frac{k_{j+1}}{N}\biggr)^{256\delta\lfloor k_{j+1}/16\rfloor}\\
&\leq \biggl(\frac{k_{j+1}}{N}\biggr)^{4\delta k_{j}}.
\end{align*}
Now, taking the union bound over all $u\in\Net_j$, we get
$$\Prob\Bigl\{\sup\limits_{u\in\Net_j}\kmax{\lfloor k_{j+1}/16\rfloor}{I^c}{|W_{u,\ell}|}\geq \tau_j\Bigr\}
\leq \biggl(\frac{k_{j+1}}{N}\biggr)^{4\delta k_{j}}|\Net_j|
\leq \Bigl(\frac{C_{\text{\tiny\ref{l: s-p net}}}k_j}{\delta N}\Bigr)^{2\delta k_j}.$$
We can assume that $N\geq {C_{\text{\tiny\ref{l: s-p net}}}}^2 \delta^{-2}k_j$, so that
$$\Bigl(\frac{C_{\text{\tiny\ref{l: s-p net}}}k_j}{\delta N}\Bigr)^{2\delta k_j}
\leq \Bigl(\frac{k_j}{N}\Bigr)^{\delta k_j}\leq \Bigl(\frac{k_t}{N}\Bigr)^{\delta k_t}
\ll \frac{1}{N^4}.$$
Thus,
$$\Prob\Bigl\{\sup\limits_{u\in\Net_j}\kmax{\lfloor k_{j+1}/16\rfloor}{I^c}{|W_{u,\ell}|}\geq \tau_j\Bigr\}
\leq \frac{1}{N^4}.$$
Summing up over $j$, repeating the same argument for nets $\Net_j'$ and applying
Proposition~\ref{p: recurrent}, we get
$$g(n,\ColorC,I)\leq
C_{\text{\tiny\ref{p: recurrent}}}\delta^{-2}\log (n)\max\limits_{i\neq j\in \ColorC}|\langle X_i,X_j\rangle|
+2C_{\text{\tiny\ref{p: recurrent}}}\delta^{-2}\sum\limits_{j=0}^{t-1}\sqrt{k_j}\tau_j$$
with probability at least $1-\frac{1}{N^3}$.
It remains to note that
for some constant $C>0$ the sum $\sum_{j=0}^{t-1}\sqrt{k_j}\tau_j$ can be estimated as
$$\sum_{j=0}^{t-1}\sqrt{k_j}\tau_j\leq CB^{1/p}\sqrt{n}\Bigl(\frac{N}{n}\Bigr)^{1/p}\sqrt{f(n,[N])}
\sum\limits_{j=0}^{t-1}2^{j/p+256\delta j/p-j/2},$$
and for a large enough constant $C_{\text{\tiny\ref{p: prob w}}}$,
the condition $\delta^{-1}=\log\frac{N}{n}\geq C_{\text{\tiny\ref{p: prob w}}}/(p-2)$
implies that $\sum_{j=0}^{t-1}2^{j/p+256\delta j/p-j/2}
\leq \sum_{j=0}^{t-1}2^{j/(2p)-j/4}\leq \frac{\widetilde C p}{p-2}$.
\end{proof}

\begin{lemma}\label{l: g counting}
Assume that $n,N>1$ satisfy
$\log\frac{N}{n}\geq C_{\text{\tiny\ref{p: prob w}}}\max\bigl(1,1/(p-2)\bigr)$.
Let $H>0$, $m\in\N$, and let $\ColorC_m^H$ be the class from the coloring of $X_1,X_2,\dots,X_N$
with threshold $H$.
Then for a universal constant $C_{\text{\tiny\ref{l: g counting}}}$ we have
\begin{align*}
f(n,\ColorC_m^H)\leq
\max\limits_{i\leq N}\|X_i\|^2
&+C_{\text{\tiny\ref{l: g counting}}}H\log^2\frac{N}{n}\,\log (n)\,\max\limits_{i\leq N}\|X_i\|\\
&+\frac{C_{\text{\tiny\ref{l: g counting}}}pB^{1/p}}{p-2}\log^2\frac{N}{n}\,\sqrt{n}
\Bigl(\frac{N}{n}\Bigr)^{1/p}\sqrt{f(n,[N])}
\end{align*}
with probability at least $1-\frac{1}{N^2}$.
\end{lemma}
\begin{proof}
Recall that
\begin{equation}\label{eq: aux 123}
f(n,\ColorC_m^H)\leq\max\limits_{i\leq N}\|X_i\|^2+2^{-N+2}\sum\limits_{I\subset [N]}g(n,\ColorC_m^H,I).
\end{equation}
Note that by Proposition~\ref{p: prob w}, together with the definition of the class $\ColorC_m^H$, we have for any $I\subset[N]$:
\begin{align}
g(n,\ColorC_m^H,I)&> C_{\text{\tiny\ref{p: prob w}}}\log^2\frac{N}{n}\,\log (n)\,\max\limits_{i\neq j\in \ColorC_m^H}|\langle X_i,X_j\rangle|
\nonumber\\
&\hspace{1cm}+\frac{C_{\text{\tiny\ref{p: prob w}}}pB^{1/p}}{p-2}\log^2\frac{N}{n}\,\sqrt{n}\Bigl(\frac{N}{n}\Bigr)^{1/p}\sqrt{f(n,[N])}
\label{eq: aux 14}
\end{align}
with probability at most $\frac{1}{N^3}$,
whence
$$\Exp\bigl|\bigl\{I\subset[N]:\,g(n,\ColorC_m^H,I)\mbox{ satisfies \eqref{eq: aux 14}}\bigr\}\bigr|
\leq \frac{2^N}{N^3}.$$
Thus, by Markov's inequality,
$$\Prob\bigl\{g(n,\ColorC_m^H,I)\mbox{ satisfies \eqref{eq: aux 14}
for at least $2^N/N$ subsets $I$}\bigr\}\leq \frac{1}{N^2}.$$
At the same time, a crude deterministic bound for $g(n,\ColorC_m^H,I)$ gives
$$g(n,\ColorC_m^H,I)\leq N\max\limits_{i\neq j\in \ColorC_m^H}|\langle X_i,X_j\rangle|
\leq NH\max\limits_{i\leq N}\|X_i\|\;\mbox{ for all $I\subset[N]$}.$$
Combining the estimates, we obtain
\begin{align*}
\Prob\Bigl\{\sum\limits_{I\subset [N]}g(n,\ColorC_m^H,I)
&\leq 2C_{\text{\tiny\ref{p: prob w}}}2^NH\log^2\frac{N}{n}\,\log (n)\,\max\limits_{i\leq N}\|X_i\|\\
&\hspace{1cm}+\frac{C_{\text{\tiny\ref{p: prob w}}}2^NpB^{1/p}}{p-2}\log^2\frac{N}{n}\,\sqrt{n}
\Bigl(\frac{N}{n}\Bigr)^{1/p}\sqrt{f(n,[N])} \Bigr\}
\geq 1-\frac{1}{N^2}.
\end{align*}
Thus, applying \eqref{eq: aux 123}, we get
\begin{align*}
f(n,\ColorC_m^H)\leq
\max\limits_{i\leq N}\|X_i\|^2
&+8C_{\text{\tiny\ref{p: prob w}}}H\log^2\frac{N}{n}\,\log (n)\,\max\limits_{i\leq N}\|X_i\|\\
&+\frac{4C_{\text{\tiny\ref{p: prob w}}}pB^{1/p}}{p-2}\log^2\frac{N}{n}\,\sqrt{n}
\Bigl(\frac{N}{n}\Bigr)^{1/p}\sqrt{f(n,[N])}
\end{align*}
with probability at least $1-\frac{1}{N^2}$.
\end{proof}

Combining the last statement with the proposition from Section~\ref{s: coloring}, we obtain the
main result of this section.

\begin{prop}\label{p: f est f}
There is a non-increasing function $\nu:(2,\infty)\to\R_+$ with the following property:
Let $N,n>1$, $p>2$, $B\geq 1$, and assume that
$\log\frac{N}{n}\geq C_{\text{\tiny\ref{p: prob w}}}\max\bigl(1,1/(p-2)\bigr)$.
Let, as before, $X$ be a centered $n$-dimensional isotropic random vector, $X_1,X_2,\dots,X_N$
be its independent copies, and assume that
$\sup\limits_{a\in\Sph^{n-1}}\Exp|\langle X,a\rangle|^p\leq B$.
Finally, let $f(n,[N])$ be defined by \eqref{eq: f def}.
Then
$$f(n,[N])\leq \nu(p)\max\limits_{i\leq N}\|X_i\|^2+\nu(p)B^{2/p}n\Bigl(\frac{N}{n}\Bigr)^{2/p}\log^4\frac{N}{n}$$
with probability at least $1-\frac{2}{n^2}$.
\end{prop}
\begin{proof}
If $n< \bigl(\frac{8p}{p-2}\bigr)^{8p/(p-2)}$ then
a crude deterministic bound for $f(n,[N])$ gives
$$f(n,[N])\leq n\max\limits_{i\leq N}\|X_i\|^2\leq
\Bigl(\frac{8p}{p-2}\Bigr)^{8p/(p-2)}\max\limits_{i\leq N}\|X_i\|^2,$$
and we obtain the statement.

Otherwise, we have
\begin{equation}\label{eq: aux n p}
n\geq \Bigl(\frac{8p}{p-2}\Bigr)^{8p/(p-2)},
\end{equation}
and both $n$ and $N$ satisfy the assumptions of
Proposition~\ref{p: prob w}. Define
$$H:=(BN)^{1/p}n^{1/2-1/p}/\log (n)$$
and let
$\chi:=\bigl\lceil\frac{8+2p}{p-2}\bigr\rceil$.
Then, by Lemma~\ref{l: g counting} and the definition of $H$, we have
\begin{align}
\chi^{-1}\sum\limits_{m=1}^\chi f(n,\ColorC_m^H)
\leq
\max\limits_{i\leq N}\|X_i\|^2
&+C_{\text{\tiny\ref{l: g counting}}}B^{1/p} \log^2\frac{N}{n}\,\sqrt{n}\Bigl(\frac{N}{n}\Bigr)^{1/p}\,
\max\limits_{i\leq N}\|X_i\|
\nonumber\\
&+\frac{C_{\text{\tiny\ref{l: g counting}}} pB^{1/p}}{p-2}\log^2\frac{N}{n}\,\sqrt{n}
\Bigl(\frac{N}{n}\Bigr)^{1/p}\sqrt{f(n,[N])}\label{eq: aux gb}
\end{align}
with probability at least $1-\frac{\chi}{N^2}$.
Now, recall that $\ColorC_m^H=\emptyset$ for all $m>\chi(\Graph_H)$,
where $\Graph_H$ is the graph defined in Section~\ref{s: coloring}, and $\chi(\Graph_H)$
is its chromatic number. By Proposition~\ref{p: chrom},
we have $\chi(\Graph_H)\leq \chi$ with probability at least $1-(BNH^{-p})^\chi n^{p/2}$.
Note that, by the assumption \eqref{eq: aux n p} on $n$, we have
$$n^{1/2-1/p}/\log(n)\geq n^{1/4-1/(2p)},$$
whence
$$(BNH^{-p})^{\chi}=\bigl(n^{p/2-1}/\log^p(n)\bigr)^{-\chi}\leq n^{\chi/2-p\chi/4}\leq \frac{1}{n^{2+p/2}}.$$
Thus, $\chi(\Graph_H)\leq \chi$ with probability at least $1-\frac{1}{n^2}$.
Note that by the definition of $f(n,\ColorC)$ and the partition $\{\ColorC_m^H\}_{m=1}^N$,
and by the Cauchy--Schwarz inequality, we have
$$f(n,[N])\leq \sum\limits_{m=1}^{\chi(\Graph_H)} f(n,\ColorC_m^H)$$
deterministically.
Therefore, in view of the above estimate of the chromatic number, we have
$$f(n,[N])\leq \sum\limits_{m=1}^\chi f(n,\ColorC_m^H)$$
with probability at least $1-\frac{1}{n^2}$. Together with the probability bound for \eqref{eq: aux gb}, it yields
\begin{align*}
\chi^{-1}f(n,[N])\leq  \max\limits_{i\leq N}\|X_i\|^2
&+C_{\text{\tiny\ref{l: g counting}}}B^{1/p} \log^2\frac{N}{n}\,\sqrt{n}\Bigl(\frac{N}{n}\Bigr)^{1/p}\,\max\limits_{i\leq N}\|X_i\|\\
&+\frac{C_{\text{\tiny\ref{l: g counting}}} pB^{1/p}}{p-2}\log^2\frac{N}{n}\,\sqrt{n}
\Bigl(\frac{N}{n}\Bigr)^{1/p}\sqrt{f(n,[N])}
\end{align*}
with probability at least $1-\frac{\chi}{N^2}-\frac{1}{n^2}$.
Solving the inequality, we obtain the result.

\end{proof}

\section{Proof of Theorem~\ref{t: main}}\label{s: proof}

The contents of this section is to a large extent based on arguments from papers \cite{ALPT10,MP14,GLPT15}.
Let us emphasize that the new ingredients --- the Sparsifying Lemma~\ref{l: s-r}
and the coloring of the sample from Section~\ref{s: coloring} --- were employed to
bound the quantity $f(n,[N])$, whereas transition from those bounds to estimating
the extreme singular values of the sample covariance matrix is well understood and covered in literature.
Nevertheless, we prefer to include all the proofs for completeness.

We start with estimating the Euclidean norm of a tail of a random vector with independent coordinates.
\begin{lemma}\label{l: tail bound}
Let $n,N\in\N$ with $n\leq N$, and let $Y=(Y_1,Y_2,\dots,Y_N)$ be a vector of independent
non-negative random variables such that $\Exp {Y_i}^q\leq B$ ($i=1,2,\dots,N$) for
some $q>1$ and $B\geq 1$. Then
$$\biggl(\sum\limits_{i=n+1}^N \kmax{i}{[N]}{Y_\ell}^2\biggr)^{1/2}
\leq C_{\text{\tiny\ref{l: tail bound}}}B^{1/q}\sqrt{n}\Bigl(\frac{N}{n}\Bigr)^{1/\min(q,2)}$$
with probability at least $1-\exp(-3n)$. Here, $C_{\text{\tiny\ref{l: tail bound}}}>0$ is a sufficiently
large universal constant.
\end{lemma}
\begin{proof}
Set $M:=\bigl(\frac{e^5BN}{n}\bigr)^{1/q}$.
By Markov's inequality,
\begin{align*}
\Prob\bigl\{|\{i\leq N:\,Y_i\geq M\}|> n\bigr\}
&=\Prob\bigl\{|\{i\leq N:\,{Y_i}^q\geq e^5BN/n\}|> n\bigr\}\\
&\leq{N\choose n}\biggl(\frac{n}{e^5N}\biggr)^n\\
&\leq\exp(-4n).
\end{align*}
We define $\widetilde Y=(\widetilde Y_1,\widetilde Y_2,\dots,\widetilde Y_N)$ as a vector of truncations of $Y_i$'s:
for every point $\omega$ of the probability space, we set
$$\widetilde Y_i(\omega):=\begin{cases}Y_i(\omega),&\mbox{if }Y_i(\omega)\leq M;\\M,&\mbox{otherwise}.\end{cases}$$
Then, from the above estimate,
\begin{align*}
\Prob\Bigl\{
\Bigl(\sum\limits_{i=n+1}^N \kmax{i}{[N]}{\widetilde Y_\ell}^2\Bigr)^{1/2}<
\Bigl(\sum\limits_{i=n+1}^N \kmax{i}{[N]}{Y_\ell}^2\Bigr)^{1/2}\Bigr\}
&=\Prob\bigl\{|\{i\leq N:\,Y_i> M\}|> n\bigr\}\\
&\leq\exp(-4n).
\end{align*}
Now, we estimate the Euclidean norm of $\widetilde Y$ using the Laplace transform.
We set $\lambda:=\frac{1}{M^2}=\bigl(\frac{n}{e^5BN}\bigr)^{2/q}$. We have
\begin{align*}
\Exp\exp(\lambda\|\widetilde Y\|^2)&=\prod\limits_{i=1}^N\Exp\exp(\lambda {\widetilde Y_i}^2)\\
&=\prod\limits_{i=1}^N\Bigl(1+\int_1^{\exp(\lambda M^2)}\Prob\bigl\{\exp(\lambda {\widetilde Y_i}^2)\geq\tau\bigr\}\,d\tau\Bigr)\\
&\leq\prod\limits_{i=1}^N\Bigl(1+\int_1^{e}\Prob\Bigl\{{\widetilde Y_i}^2\geq\frac{\tau-1}{e\lambda}\Bigr\}\,d\tau\Bigr)\\
&\leq\prod\limits_{i=1}^N\Bigl(1+e\lambda\int_0^{(e-1)/(e\lambda)}\Prob\bigl\{{\widetilde Y_i}^2\geq u\bigr\}\,du\Bigr)\\
&\leq \prod\limits_{i=1}^N\bigl(1+e\lambda\Exp{\widetilde Y_i}^2\bigr).
\end{align*}
First, assume that $q\geq 2$. Then $\Exp {\widetilde Y_i}^2\leq B^{2/q}$, and we get
$$\Exp\exp(\lambda\|\widetilde Y\|^2)
\leq\bigl(1+eB^{2/q}\lambda\bigr)^N\leq \exp(eB^{2/q}\lambda N).
$$
Otherwise, if $q<2$ then $\Exp{\widetilde Y_i}^2\leq M^{2-q}\Exp {\widetilde Y_i}^q\leq B\lambda^{q/2-1}$,
whence
$$
\Exp\exp(\lambda\|\widetilde Y\|^2)
\leq\bigl(1+eB\lambda^{q/2}\bigr)^N\leq \exp(eB\lambda^{q/2} N).
$$
Thus, denoting $r:=\min(q,2)$, we get
$$\Exp\exp(\lambda\|\widetilde Y\|^2)\leq\exp\bigl(eB^{2/q-2/r+1}\lambda^{r/2}N\bigr).$$
Hence, by Markov's inequality,
\begin{align*}
\Prob\bigl\{\|\widetilde Y\|\geq e^6B^{1/q}N^{1/r}n^{1/2-1/r}\bigr\}
&\leq \exp\bigl(eB^{2/q-2/r+1}\lambda^{r/2}N-e^{12}B^{2/q}\lambda N^{2/r}n^{1-2/r}\bigr)\\
&\leq\exp\bigl(-(e^7-e)B^{r/q}\lambda^{r/2}N\bigr)\\
&\leq\exp(-4n).
\end{align*}
Finally, we get
\begin{align*}
&\Prob\Bigl\{\Bigl(\sum\limits_{i=n+1}^N \kmax{i}{[N]}{Y_\ell}^2\Bigr)^{1/2}>e^6B^{1/q}N^{1/r}n^{1/2-1/r}\Bigr\}\\
&\leq\Prob\Bigl\{
\Bigl(\sum\limits_{i=n+1}^N \kmax{i}{[N]}{\widetilde Y_\ell}^2\Bigr)^{1/2}<
\Bigl(\sum\limits_{i=n+1}^N \kmax{i}{[N]}{Y_\ell}^2\Bigr)^{1/2}\Bigr\}
+\Prob\bigl\{\|\widetilde Y\|> e^6B^{1/q}N^{1/r}n^{1/2-1/r}\bigr\}\\
&\leq 2\exp(-4n)\\
&\leq\exp(-3n).
\end{align*}
\end{proof}
\begin{rem}
The above lemma is similar, but not identical to \cite[Lemma~4.4]{GLPT15}, which was proved under slightly
different assumptions, and using different arguments.
\end{rem}

\begin{prop}\label{p: rademacher}
Let $X_1,X_2,\dots,X_N$ be i.i.d.\ centered $n$-dimensional isotropic random vectors,
and assume that for some $p>2$ and $B\geq 1$ we have
$$\Exp|\langle X_i,a\rangle|^p\leq B$$
for all $a\in\Sph^{n-1}$. Further, let $r_1,r_2,\dots,r_N$ be Rademacher ($\pm1$) random variables
jointly independent with $X_1,X_2\dots,X_N$. Then
$$\sup\limits_{a\in\Sph^{n-1}}\Bigl|\sum\limits_{i=1}^N r_i \langle X_i,a\rangle^2\Bigr|
\leq C_{\text{\tiny\ref{p: rademacher}}}B^{2/p}n\Bigl(\frac{N}{n}\Bigr)^{2/\min(p,4)}+2f(n,[N])$$
with probability at least $1-2^{-n}$. Here,
$f$ is defined by \eqref{eq: f def}, and
$C_{\text{\tiny\ref{p: rademacher}}}>0$ is a universal constant. 
\end{prop}
\begin{proof}
Define a random operator $T:\R^n\to\R^n$ by
$$Ta:=\sum\limits_{i=1}^N r_i \langle X_i,a\rangle X_i,\;\;a\in\Sph^{n-1}.$$
Let $\Net$ be a Euclidean $1/4$-net on $\Sph^{n-1}$ of cardinality at most $9^n$.
Then, applying Lemma~\ref{l: quadratic net}, we obtain
$$\sup\limits_{a\in\Sph^{n-1}}|\langle Ta,a\rangle|
\leq 2\sup\limits_{a\in\Net}|\langle Ta,a\rangle|=2\sup\limits_{a\in\Net}\Bigl|\sum\limits_{i=1}^N r_i \langle X_i,a\rangle^2\Bigr|.$$
Next, for every $a\in\Sph^{n-1}$ let $\sigma_a$
be a random permutation on $[N]$ measurable with respect to the $\sigma$-algebra
generated by $X_1,X_2,\dots,X_N$, such that
$$\langle X_{\sigma_a(1)},a\rangle^2\geq \langle X_{\sigma_a(2)},a\rangle^2
\geq\dots\geq \langle X_{\sigma_a(N)},a\rangle^2.$$
Thus, $\kmax{i}{[N]}\langle X_\ell,a\rangle^2=\langle X_{\sigma_a(i)},a\rangle^2$
for all $i=1,2,\dots,N$. We have
\begin{align*}
\sup\limits_{a\in\Net}\Bigl|\sum\limits_{i=1}^N r_i \langle X_i,a\rangle^2\Bigr|
&\leq \sup\limits_{a\in\Sph^{n-1}}\sum\limits_{i=1}^n \kmax{i}{[N]}\langle X_\ell,a\rangle^2
+\sup\limits_{a\in\Net}\Bigl|\sum\limits_{i=n+1}^N r_{\sigma_a(i)} \langle X_{\sigma_a(i)},a\rangle^2\Bigr|\\
&=f(n,[N])+\sup\limits_{a\in\Net}\Bigl|\sum\limits_{i=n+1}^N r_{\sigma_a(i)} \langle X_{\sigma_a(i)},a\rangle^2\Bigr|.
\end{align*}
In view of the fact that $\sigma_a$ is independent from $r_1,r_2,\dots,r_N$, it remains to prove that
$$\sup\limits_{a\in\Net}\Bigl|\sum\limits_{i=n+1}^N r_i\,\kmax{i}{[N]}\langle X_\ell,a\rangle^2\Bigr|\leq
CB^{2/p}n\Bigl(\frac{N}{n}\Bigr)^{2/\min(p,4)}$$
with probability at least $1-2^{-n}$ for a sufficiently large universal constant $C>0$.
Fix for a moment $a\in\Net$ and define a random vector $Z^a\in\R^{[N]\setminus [n]}$ by
$$Z^a_i:=\kmax{i}{[N]}\langle X_\ell,a\rangle^2,\;\;i=n+1,\dots,N.$$
Note that $Z^a$ and $r_1,r_2,\dots,r_N$ are jointly independent.
Applying Hoeffding's inequality, we get
$$\Bigl|\sum\limits_{i=n+1}^N r_i Z_i^a\Bigr|\leq 4\sqrt{n}\|Z^a\|$$
with probability at least $1-2\exp(-8n)$ (see Lemma~\ref{l: hoeff}).
At the same time, applying Lemma~\ref{l: tail bound} with $Y_i:=\langle X_i,a\rangle^2$ ($i=1,2,\dots,N$)
and $q:=p/2$, we get
$$\Prob\Bigl\{\|Z^a\|> C_{\text{\tiny\ref{l: tail bound}}}B^{2/p}\sqrt{n}\Bigl(\frac{N}{n}\Bigr)^{2/\min(p,4)}\Bigr\}
\leq \exp(-3n).$$
Combining the last two estimates, we obtain
\begin{align*}
\Prob\Bigl\{\Bigl|\sum\limits_{i=n+1}^N r_i \kmax{i}{[N]}\langle X_\ell,a\rangle^2\Bigr|>
4C_{\text{\tiny\ref{l: tail bound}}}B^{2/p}n\Bigl(\frac{N}{n}\Bigr)^{2/\min(p,4)}\Bigr\}
&\leq 2\exp(-8n)+\exp(-3n)\\
&\leq 18^{-n}.
\end{align*}
Taking the union bound over all $a\in\Net$,
we obtain the desired result.
\end{proof}

In the next statement, we combine Proposition~\ref{p: rademacher}
with a standard symmetrization argument. Let us note once more that at this point
our proof essentially follows the argument of \cite{MP14} and \cite{GLPT15}.

\begin{prop}\label{p: deviation}
Let $X_1,X_2,\dots,X_N$ be i.i.d.\ centered isotropic random vectors,
such that for some $p> 2$ and $B\geq 1$ we have
$$\sup\limits_{a\in\Sph^{n-1}}\Exp|\langle X_i,a\rangle|^{p}\leq B.$$
Then
$$\sup\limits_{a\in\Sph^{n-1}}\Bigl|\sum\limits_{i=1}^N \langle X_i,a\rangle^2-N\Bigr|
\leq C_{\text{\tiny\ref{p: deviation}}}B^{2/p}n\Bigl(\frac{N}{n}\Bigr)^{2/\min(p,4)}+4f(n,[N])$$
with probability at least $1-2^{2-n}$. Here, $C_{\text{\tiny\ref{p: deviation}}}>0$ is a universal constant.
\end{prop}
\begin{proof}
Let $X_1',X_2',\dots,X_N'$ be jointly independent copies of $X_1,X_2,\dots,X_N$,
and for every $a\in\Sph^{n-1}$ denote
$$\xi_a:=\sum\limits_{i=1}^N\langle X_i,a\rangle^2-N;\;\;\xi_a':=\sum\limits_{i=1}^N\langle X_i',a\rangle^2-N.$$
Then the variable
$$\sup\limits_{a\in\Sph^{n-1}}|\xi_a-\xi_a'|=\sup\limits_{a\in\Sph^{n-1}}\Bigl|
\sum\limits_{i=1}^N\bigl(\langle X_i,a\rangle^2-\langle X_i',a\rangle^2\bigr)\Bigr|$$
has the same distribution as
$$\sup\limits_{a\in\Sph^{n-1}}\Bigl|
\sum\limits_{i=1}^N r_i\bigl(\langle X_i,a\rangle^2-\langle X_i',a\rangle^2\bigr)\Bigr|,$$
where $r_1,r_2,\dots,r_N$ are Rademacher random variables jointly independent with the vectors $X_i,X_i'$.
Hence, for any $t\geq 0$ we have
$$\Prob\bigl\{\sup\limits_{a\in\Sph^{n-1}}|\xi_a-\xi_a'|\geq t\bigr\}
\leq 2\Prob\Bigl\{\sup\limits_{a\in\Sph^{n-1}}\Bigl|\sum\limits_{i=1}^N r_i\langle X_i,a\rangle^2\Bigr|\geq\frac{t}{2}\Bigr\}.$$
Applying Proposition~\ref{p: rademacher},
we obtain
\begin{equation}\label{eq: aux 9472}
\Prob\Bigl\{\sup\limits_{a\in\Sph^{n-1}}|\xi_a-\xi_a'|>
2C_{\text{\tiny\ref{p: rademacher}}}B^{2/p}n\Bigl(\frac{N}{n}\Bigr)^{2/\min(p,4)}+4f(n,[N])\Bigr\}
\leq 2^{1-n}.
\end{equation}
Finally, note that for any $a\in\Sph^{n-1}$ we have
\begin{align*}
\Exp|\xi_a|&\leq \Exp|\xi_a-\xi_a'|\\
&=\Exp\Bigl|\sum\limits_{i=1}^N r_i\bigl(\langle X_i,a\rangle^2-\langle X_i',a\rangle^2\bigr)\Bigr|\\
&\leq 2\Exp\Bigl|\sum\limits_{i=1}^N r_i\langle X_i,a\rangle^2\Bigr|\\
&\leq 2\Exp\Bigl(\sum\limits_{i=1}^N \langle X_i,a\rangle^4\Bigr)^{1/2},
\end{align*}
whence, by the Minkowski inequality,
\begin{align*}
\Exp|\xi_a|&\leq 2\Exp\Bigl(\sum\limits_{i=1}^N |\langle X_i,a\rangle|^{\min(p,4)}\Bigr)^{2/\min(p,4)}\\
&\leq 2\Bigl(\sum\limits_{i=1}^N \Exp|\langle X_i,a\rangle|^{\min(p,4)}\Bigr)^{2/\min(p,4)}\\
&\leq 2B^{2/p}N^{2/\min(p,4)}.
\end{align*}
Therefore, by Markov's inequality
$$\Med|\xi_a|\leq 4B^{2/p}N^{2/\min(p,4)},\;\;a\in\Sph^{n-1}.$$
Combining this with a standard estimate
$$\Prob\bigl\{\sup\limits_{a\in\Sph^{n-1}}|\xi_a|\geq t+M\bigr\}
\leq 2\Prob\bigl\{\sup\limits_{a\in\Sph^{n-1}}|\xi_a-\xi_a'|\geq t\bigr\},\;\;t>0,$$
where $M:=\sup\limits_{a\in\Sph^{n-1}}\Med|\xi_a'|$, and with \eqref{eq: aux 9472},
we obtain the result.
\end{proof}

\begin{proof}[Proof of Theorem~\ref{t: main}]
Let $X$ be a centered isotropic random vector in $\R^n$,
such that
$$\sup\limits_{a\in\Sph^{n-1}}\Exp|\langle X,a\rangle|^p\leq B$$
for some $p>2$ and $B\geq 1$. Also, let $X_1,X_2,\dots,X_N$ be its independent copies.
As before, we denote by $\CovM_N$ the sample covariance matrix for $X_1,X_2,\dots,X_N$,
and by $A_N$ --- the $N\times n$ random matrix with rows $X_1,X_2,\dots,X_N$.
If $N$ is bounded by a function of $p$, we can apply a trivial estimate to get the result.
So, further we assume that  
$$N>2\exp\bigl(3C_{\text{\tiny\ref{p: prob w}}}\max(1,1/(p-2))\bigr).$$

First, suppose that $\log\frac{N}{n}\geq C_{\text{\tiny\ref{p: prob w}}}\max\bigl(1,1/(p-2)\bigr)$.
We have
\begin{align*}
\|\CovM_N-\Id_n\|_{2\to 2}&=\sup\limits_{a\in\Sph^{n-1}}|\langle\CovM_N a-a,a\rangle|\\
&=N^{-1}\sup\limits_{a\in\Sph^{n-1}}|\langle {A_N}^T A_N a,a\rangle-N|\\
&=N^{-1}\sup\limits_{a\in\Sph^{n-1}}\Bigl|\sum\limits_{i=1}^N \langle X_i,a\rangle^2-N\Bigr|.
\end{align*}
Hence, by Proposition~\ref{p: deviation}, we have
$$\|\CovM_N-\Id_n\|_{2\to 2}\leq
C_{\text{\tiny\ref{p: deviation}}}B^{2/p}n\Bigl(\frac{N}{n}\Bigr)^{2/\min(p,4)}+4f(n,[N])$$
with probability at least $1-2^{2-n}$.
Then apply Proposition~\ref{p: f est f}.

\bigskip

Now, if $\log\frac{N}{n}< C_{\text{\tiny\ref{p: prob w}}}\max\bigl(1,1/(p-2)\bigr)$
then it is enough to check that
$$\bigl\|(A_N)^TA_N\bigr\|_{2\to 2}\leq
\widetilde \nu(p)\max\limits_{i\leq N}\|X_i\|^2+
\widetilde \nu(p)B^{2/p}N$$
with probability $\geq 1-\frac{1}{n}$ for some non-increasing function $\nu(p):(2,\infty)\to\R_+$.
Set $n_0:=\bigl\lfloor N/\exp\bigl(C_{\text{\tiny\ref{p: prob w}}}\max(1,1/(p-2))\bigr)\bigr\rfloor$ and
consider an arbitrary partition $\{J_m\}_{m=1}^{\lceil n/n_0\rceil}$ of $[n]$ with
$\max\limits_{m}|J_m|\leq n_0$.
For every $m\leq \lceil n/n_0\rceil$, we let $A_N^m$ be the $[N]\times J_m$-submatrix of $A_N$.
Note that the rows of $A_N^m$ are isotropic (in $\R^{J_m}$), i.i.d., and satisfy the $p$-th moment condition
for one-dimensional projections.
Moreover, the ratio of the numbers of rows and columns in every matrix $A_N^m$
satisfies the assumptions of the first part of the theorem. Hence, by the above argument,
for any $m\leq  \lceil n/n_0\rceil$ we have
$$\bigl\|\bigl(A_N^m\bigr)^T A_N^m\bigr\|_{2\to 2}
\leq \nu'(p)\max\limits_{i\leq N}\|X_i\|^2+\nu'(p)B^{2/p}N$$
with probability at least $1-\frac{2}{{n_0}^2}-2^{2-n_0}$.
Since
$$\bigl\|(A_N)^TA_N\bigr\|_{2\to 2}\leq \sum_{m=1}^{\lceil n/n_0\rceil}\bigl\|(A_N^m)^T A_N^m\bigr\|_{2\to 2},$$
we obtain the result by taking the union bound.

Finaly, we may
use a standard linear algebraic argument to pass from isotropic distributions to all distributions from
the class $\Class(n,p,B)$.
\end{proof}

\bigskip

Let us briefly discuss optimality of the result obtained. As we already mentioned,
for $p=4$ the log-factor which appears in our bound in Theorem~\ref{t: main}, seems excessive.
In the range $2<p< 4$, the situation is more unclear to us.
We do not know whether the estimate for the difference $\|\CovM_N-\Id_n\|_{2\to 2}$
(for isotropic distributions) can be improved if we assume a strong concentration for the vector norm.
Let us formulate the problem in a more precise form:
\begin{problem}
Let $2<p<4$ and assume that $X$ is a centered $n$-dimensional isotropic random vector such that $\|X\|\leq C\sqrt{n}$ a.s.\
and $\sup\limits_{y\in\Sph^{n-1}}\Exp|\langle X,y\rangle|^p\leq C$ for a large universal constant $C>0$.
Let $N\geq n$ and let $X_1,X_2,\dots,X_N$ be independent copies of $X$.
As before, let $\CovM_N$ be the sample covariance matrix with respect to $X_1,X_2,\dots,X_N$.
Is it true that
$$\|\CovM_N-\Id_n\|_{2\to 2}\leq K\sqrt{\frac{n}{N}}$$
with probability close to one, where $K$ depends only on $p$?
\end{problem}
Let us remark that an example of P.~Yaskov \cite{Y15} which provides {\it upper} bounds for the smallest 
eigenvalue of $\CovM_N$ for certain isotropic distributions, does not resolve the above problem in negative as
an essential assumption in \cite{Y15} is a growth condition on the vector norm.

\bigskip

{\bf Acknowledgement.}
I would like to thank Nicole Tomczak-Jaegermann for her support, Alexander Litvak
for clarifying some arguments from \cite{GLPT15}, and
Shahar Mendelson and Ramon van Handel for a fruitful discussion.


\begin{thebibliography}{99}

\bibitem{ALPT10}
{
R. Adamczak, A.E. Litvak, A. Pajor, N. Tomczak-Jaegermann,
Quantitative estimates of the convergence of the empirical covariance matrix in log-concave ensembles,
J. Amer. Math. Soc. {\bf 23} (2010), no.~2, 535--561. MR2601042
}

\bibitem{ALPT11}
{
R. Adamczak, A.E. Litvak, A. Pajor, N. Tomczak-Jaegermann,
Sharp bounds on the rate of convergence of the empirical covariance matrix,
C. R. Math. Acad. Sci. Paris {\bf 349} (2011), no.~3-4, 195--200. MR2769907
}

\bibitem{BY93}
{
Z. D. Bai\ and\ Y. Q. Yin, Limit of the smallest eigenvalue of a large-dimensional sample covariance matrix,
Ann. Probab. {\bf 21} (1993), no.~3, 1275--1294. MR1235416
}

\bibitem{BL08a}
{
P. J. Bickel\ and\ E. Levina, Covariance regularization by thresholding,
Ann. Statist. {\bf 36} (2008), no.~6, 2577--2604. MR2485008
}

\bibitem{BL08b}
{
P. J. Bickel\ and\ E. Levina, Regularized estimation of large covariance matrices,
Ann. Statist. {\bf 36} (2008), no.~1, 199--227. MR2387969
}

\bibitem{B14}
{
J. Bourgain, An improved estimate in the restricted isometry problem,
in {\it Geometric aspects of functional analysis}, 65--70, Lecture Notes in Math., 2116, Springer, Cham. MR3364679
}

\bibitem{B96}
{
J. Bourgain, Random points in isotropic convex sets, in {\it Convex geometric analysis (Berkeley, CA, 1996)},
53--58, Math. Sci. Res. Inst. Publ., 34, Cambridge Univ. Press, Cambridge. MR1665576
}

\bibitem{CT15}
{
D. Chafa\"i, K. Tikhomirov,
On the convergence of the extremal eigenvalues of empirical covariance matrices with dependence,
Preprint. arXiv:1509.02231
}

\bibitem{FS10}
{
O. N. Feldheim\ and\ S. Sodin, A universality result for the smallest eigenvalues of certain sample covariance matrices,
Geom. Funct. Anal. {\bf 20} (2010), no.~1, 88--123. MR2647136
}

\bibitem{GLPT15}
{
O. Gu\'edon, A.E. Litvak, A. Pajor, N. Tomczak-Jaegermann,
On the interval of fluctuation of the singular values of random matrices,
J. Eur. Math. Soc. (JEMS), to appear.
}

\bibitem{GLPT14}
{
O. Gu\'edon, A.E. Litvak, A. Pajor, N. Tomczak-Jaegermann,
Restricted isometry property for random matrices with heavy-tailed columns,
C. R. Math. Acad. Sci. Paris {\bf 352} (2014), no.~5, 431--434. MR3194251
}

\bibitem{HR}
{
I. Haviv, O. Regev,
The Restricted Isometry Property of Subsampled Fourier Matrices, arXiv:1507.01768
}

\bibitem{Hoeffding}
{
W. Hoeffding, Probability inequalities for sums of bounded random variables,
J. Amer. Statist. Assoc. {\bf 58} (1963), 13--30. MR0144363
}

\bibitem{J01}
{
I. M. Johnstone, On the distribution of the largest eigenvalue in principal components analysis,
Ann. Statist. {\bf 29} (2001), no.~2, 295--327. MR1863961
}

\bibitem{KLS95}
{
R. Kannan, L. Lov\'asz\ and\ M. Simonovits, Random walks and an $O\sp *(n\sp 5)$ volume algorithm for convex bodies,
Random Structures Algorithms {\bf 11} (1997), no.~1, 1--50. MR1608200
}

\bibitem{Kesten}
{
H. Kesten, A sharper form of the Doeblin-L\'evy-Kolmogorov-Rogozin inequality for concentration functions,
Math. Scand. {\bf 25} (1969), 133--144. MR0258095
}

\bibitem{KM15}
{
V. Koltchinskii\ and\ S. Mendelson, Bounding the smallest singular value of a random matrix without concentration,
Int. Math. Res. Not. IMRN {\bf 2015}, no.~23, 12991--13008. MR3431642
}

\bibitem{LPRT05}
{
A. E. Litvak, A. Pajor, M. Rudelson, N. Tomczak-Jaegermann,
Smallest singular value of random matrices and geometry of random polytopes,
Adv. Math. {\bf 195} (2005), no.~2, 491--523. MR2146352
}

\bibitem{MP12}
{
S. Mendelson\ and\ G. Paouris,
On generic chaining and the smallest singular value of random matrices with heavy tails,
J. Funct. Anal. {\bf 262} (2012), no.~9, 3775--3811. MR2899978
}

\bibitem{MP14}
{
S. Mendelson\ and\ G. Paouris, On the singular values of random matrices,
J. Eur. Math. Soc. (JEMS) {\bf 16} (2014), no.~4, 823--834. MR3191978
}

\bibitem{PY14}
{
N. S. Pillai\ and\ J. Yin, Universality of covariance matrices, Ann. Appl. Probab. {\bf 24} (2014), no.~3, 935--1001. MR3199978
}

\bibitem{R99}
{
M. Rudelson, Random vectors in the isotropic position, J. Funct. Anal. {\bf 164} (1999), no.~1, 60--72. MR1694526
}

\bibitem{RV08}
{
M. Rudelson\ and\ R. Vershynin, On sparse reconstruction from Fourier and Gaussian measurements,
Comm. Pure Appl. Math. {\bf 61} (2008), no.~8, 1025--1045. MR2417886
}

\bibitem{RV10}
{
M. Rudelson\ and\ R. Vershynin, Non-asymptotic theory of random matrices: extreme singular values,
in {\it Proceedings of the International Congress of Mathematicians. Volume III}, 1576--1602,
Hindustan Book Agency, New Delhi. MR2827856
}

\bibitem{RV09}
{
M. Rudelson\ and\ R. Vershynin,
Smallest singular value of a random rectangular matrix,
Comm. Pure Appl. Math. {\bf 62} (2009), no.~12, 1707--1739. MR2569075
}

\bibitem{SV13}
{
N. Srivastava\ and\ R. Vershynin, Covariance estimation for distributions with $2+\varepsilon$ moments,
Ann. Probab. {\bf 41} (2013), no.~5, 3081--3111. MR3127875
}

\bibitem{T_AiM}
{
K. Tikhomirov, The limit of the smallest singular value of random matrices with i.i.d. entries,
Adv. Math. {\bf 284} (2015), 1--20. MR3391069
}

\bibitem{T_IJM}
{
K. E. Tikhomirov, The smallest singular value of random rectangular matrices with no moment assumptions on entries,
Israel J. Math. {\bf 212} (2016), no.~1, 289--314. MR3504328
}

\bibitem{V12a}
{
R. Vershynin, How close is the sample covariance matrix to the actual covariance matrix?,
J. Theoret. Probab. {\bf 25} (2012), no.~3, 655--686. MR2956207
}

\bibitem{V12b}
{
R. Vershynin, Introduction to the non-asymptotic analysis of random matrices, in {\it Compressed sensing},
210--268, Cambridge Univ. Press, Cambridge. MR2963170
}

\bibitem{Y14}
{
P. Yaskov, Lower bounds on the smallest eigenvalue of a sample covariance matrix,
Electron. Commun. Probab. {\bf 19} (2014), no. 83, 10 pp. MR3291620
}

\bibitem{Y15}
{
P. Yaskov, Sharp lower bounds on the least singular value of a random matrix without the fourth moment condition,
Electron. Commun. Probab. {\bf 20} (2015), no. 44, 9 pp. MR3358966
}

\bibitem{YBK88}
{
Y. Q. Yin, Z. D. Bai\ and\ P. R. Krishnaiah, On the limit of the largest eigenvalue of the large-dimensional sample covariance matrix,
Probab. Theory Related Fields {\bf 78} (1988), no.~4, 509--521. MR0950344
}

\end{thebibliography}
\end{document}